\documentclass[11pt]{amsart}

\usepackage[latin1,utf8]{inputenc}
\usepackage[T1]{fontenc} 
\usepackage{textcomp} 
\usepackage[english]{babel}
\usepackage[autolanguage]{numprint} 
\usepackage{hyperref} 
\usepackage{graphicx} 
\usepackage{verbatim} 
\usepackage[all]{xy}
\usepackage{float}
\usepackage{lmodern}
\usepackage{mathrsfs}
\usepackage{amsmath,amssymb,amsthm} 
\usepackage{tikz,tkz-tab} 

\usepackage{faktor}
\usepackage{enumitem}
\usepackage{xcolor}

\setenumerate[0]{label=(\alph*)}

\newtheorem{theorem}{Theorem}[section]

\newtheorem{lemma}[theorem]{Lemma}
\newtheorem{proposition}[theorem]{Proposition}

\newtheorem{corollary}[theorem]{Corollary}

\theoremstyle{remark}
\newtheorem{definition}[theorem]{Definition}

\newtheorem*{remark}{Remark}
\newtheorem*{remarks}{Remarks}

\renewcommand\epsilon\varespilon 
\renewcommand\phi\varphi 
\renewcommand\rho\varrho 
\renewcommand\theta\vartheta 

\renewcommand\subset{\subseteq}

\newcommand\AAA{\mathcal{A}}
\newcommand\AND{\quad\textrm{and}\quad}
\newcommand\Adj{\mathrm{Adj}}
\newcommand\DD{\Delta} 
\newcommand\bP{\mathbf{P}} 
\newcommand\bL{\mathbf{L}} 
\newcommand\ba{\mathbf{a}}  
\newcommand\bb{\mathbf{b}}  
\newcommand\bx{\mathbf{x}}  
\newcommand\by{\mathbf{y}}  
\newcommand\bz{\mathbf{z}}  
\newcommand\bv{\mathbf{v}}  
\newcommand\bw{\mathrm{w}}  
\newcommand\bN{\mathbb{N}}
\newcommand\bR{\mathbb{R}}
\newcommand\bQ{\mathbb{Q}}
\newcommand\bZ{\mathbb{Z}}

\newcommand\byt{\by}  
\newcommand\bzt{\bz}  
\newcommand\bwt{\widetilde{\bw}}

\newcommand\CCC{\mathcal{C}} 
\newcommand\CP{\mathrm{P}} 
\newcommand\CL{\mathrm{L}} 
\newcommand\cont{\mathrm{cont}}  
\newcommand\ee{\varepsilon}

\newcommand\GL{\mathrm{GL}_2}
\newcommand\GrO{\mathcal{O}} 
\newcommand\hCL{\widehat{\mathrm{L}}} 
\newcommand\hlambda{\widehat{\lambda}} 
\newcommand\lambdaL{\hlambda_{\min}}
\newcommand\homega{\widehat{\omega}} 

\newcommand\Id{\mathrm{Id}}
\newcommand\ie{\textsl{i.e. }} 

\newcommand{\indiceGauche}[2]{{\vphantom{#2}}^{#1}\!\!\!#2} 
\newcommand\Mat{\mathrm{Mat}_{2\times2}}
\newcommand\MM{\mathcal{M}}
\newcommand{\norm}[1]{\|#1\|} 
\newcommand{\normI}[1]{\|#1\|_\infty} 
\newcommand\Sp{\textrm{spec}} 
\newcommand{\Span}[1]{\langle\,#1\rangle}
\newcommand\Sturm{{\emph Sturm}} 
\newcommand\Tr{\mathrm{Tr}}
\newcommand{\transpose}[1]{\indiceGauche{t\;\;}{#1}}
\newcommand\UU{\mathrm{U}}

\newcommand\pu{\underline{\psi}} 
\newcommand\po{\overline{\psi}} 

\newcommand\thu{\underline{\theta}} 
\newcommand\tho{\overline{\theta}} 

\newcommand\us{\mathbf{s}} 
\newcommand\usFibo{\mathbf{1}} 
\newcommand\say[1]{`#1'}

\newcommand{\Ack}{{
  \footnotesize

  \textbf{Acknowledgements}: The author is very grateful to Stéphane Fischler and Damien Roy for giving him a lot of feedback on this work.
}}

\title[Exponents of Diophantine approximation in dimension two]{Exponents of Diophantine approximation in dimension two for a general class of numbers}  
\author{Anthony Po\"els}
\address{
   Department of Mathematics\\
   University of Ottawa\\
   150 Louis Pasteur\\
   Ottawa, Ontario K1N 6N5, Canada}
\email{anthony.poels@uottawa.ca}

\thanks{Work of the author partially supported by NSERC}


\begin{document} 
\maketitle

\begin{abstract}
We study the Diophantine properties of a new class of transcendental real numbers which contains, among others, Roy's extremal numbers, Bugeaud-Laurent Sturmian continued fractions, and more generally the class of Sturmian type numbers. We compute, for each real number $\xi$ of this set, several exponents of Diophantine approximation to the pair $(\xi,\xi^2)$, together with $\omega_2^*(\xi)$ and $\homega_2^*(\xi)$, the so-called ordinary and uniform exponent of approximation to $\xi$ by algebraic numbers of degree $\leq 2$. As an application, we get new information on the set of values taken by $\homega_2^*$ at transcendental numbers, and we give a partial answer to a question of Fischler about his exponent $\beta_0$.
\end{abstract}

\section{Introduction}

Given a real number $\xi$, we are interested in the following six classical exponents of Diophantine approximation: the exponent $\lambda_2(\xi)$ of simultaneous rational approximation to $\xi$ and $\xi^2$, the dual exponent $\omega_2(\xi)$, the exponent $\omega_2^*(\xi)$ of approximation to $\xi$ by algebraic numbers of degree at most $2$, and the corresponding uniform exponents $\hlambda_2(\xi)$, $\homega_2(\xi)$, $\homega_2^*(\xi)$ (the precise definitions are recalled in the next section, see also \cite{bugeaud2015exponents} and \cite{bugeaud2005exponentsSturmian}). By \cite[Theorem 2.3]{bugeaud2005exponentsSturmian}, for almost all real numbers $\xi$ (with respect to the Lebesgue measure), we have
\begin{align}
    \label{eq: valeur générique 6 expos}
    \lambda_2(\xi)=\hlambda_2(\xi)=\frac{1}{2}\AND \omega_2(\xi)=\homega_2(\xi)=\omega_2^*(\xi)=\homega_2^*(\xi) = 2.
\end{align}
Moreover, if $\xi$ is algebraic of degree at least $3$, then \eqref{eq: valeur générique 6 expos} still holds as a consequence of Schmidt's subspace Theorem (see \cite[Theorem 2.4]{bugeaud2005exponentsSturmian}). There are currently few explicit families of transcendental numbers of which all exponents are known. Under the condition $\lambda_2(\xi) \leq 1$ (which excludes Liouville numbers, see \cite[Corollary 5.4]{bugeaud2015exponents}), we have Roy's extremal numbers \cite{roy2004approximation}  and Fibonacci type numbers \cite{roy2007two}, Bugeaud and Laurent Sturmian continued fractions \cite{bugeaud2005exponentsSturmian}, and more generally the class of Sturmian type numbers  \cite{poels2017exponents} which generalizes the two last families. In some cases combinatorics on words provides numbers for which the six exponents can be computed. Given a word $w$ written on the alphabet of positive integers we associate the real number $\xi_w=[0;w]$ whose partial quotients are successively $0$ and the letters of $w$. Some combinatorial properties of $w$ translate into Diophantine properties of $\xi_w$. For example, in \cite{allouche2001transcendence} Allouche, Davison, Queffélec, and Zamboni proved that when $w$ is a Sturmian or quasi-Sturmian sequence (see \cite{lothaire1983combinatorics} and \cite{cassaigne1997sequences} for the definitions), then $\omega_2^*(\xi_w) > 2$, and thus $\xi$ is transcendental. Bugeaud and Laurent \cite{bugeaud2005exponentsSturmian} studied in depth the special case of Sturmian characteristic words. Their work generalizes a previous construction of Roy based on the Fibonacci word \cite{roy2003approximation}. To state their result, let us recall some definitions. Fix an alphabet $\AAA:=\{a,b\}$, where $a,b$ are two distinct positive integers, and let $\AAA^*$ denote the monoid of finite words on $\AAA$ for the concatenation. Given an infinite sequence $\us=(s_k)_{k\geq 1}$ of positive integers or an irrational number $\phi\in(0,1)$ with continued fraction expansion $\phi=[0;s_1,s_2,\dots]$, we define recursively a sequence of finite words $(w_k)_{k\geq 0}$ in $\AAA^*$ by
\begin{align*}
    w_0 = b,\quad w_1=b^{s_1-1}a \AND w_{k+1} = w_k^{s_{k+1}}w_{k-1} \quad (k\geq 1).
\end{align*}
This sequence converges to an infinite word $w_\phi = \lim_{k\rightarrow\infty} w_k$ called the \textsl{Sturmian characteristic word of slope $\phi$ on $\AAA=\{a,b\}$}. These words are important in combinatorics on words, see for examples \cite{lothaire1983combinatorics}, \cite{de1994some}, \cite{de1997sturmian}. We associate to $w_\phi$ the real numbers $\xi_\phi:= \xi_{w_\phi}$ and
\begin{equation}
    \label{eq: def sigma}
    \sigma(\us) := \liminf_{k\rightarrow+\infty}\frac{1}{[s_{k+1};s_k,\dots,s_1]}.
\end{equation}

Note that $\sigma(\us)\leq 1/\gamma$, where $\gamma=[1;1,\cdots] = (1+\sqrt 5)/2$ denotes the golden ratio. Bugeaud and Laurent proved the following result \cite[Theorem 3.1]{bugeaud2005exponentsSturmian}.

\begin{theorem}[Bugeaud-Laurent, 2005]
\label{thm:bugeaud-laurent,sturm}
Let $\phi=[0;s_1,s_2,\dots]\in [0,1]\setminus\bQ$ and let $\sigma=\sigma(\us)$. Then
    \begin{align*}
        \begin{array}{ll}
           \homega_2(\xi_\phi) = \homega_2^*(\xi_\phi) = 2+\sigma, & \displaystyle \quad  \hlambda_2(\xi_\phi) = \frac{1+\sigma}{2+\sigma},\\
           \displaystyle \omega_2(\xi_\phi) = \omega_2^*(\xi_\phi) = \frac{2}{\sigma}+1, & \quad \lambda_2(\xi_\phi)=1.
       \end{array}
    \end{align*}
\end{theorem}
The set of values taken by $2+1/\sigma(\us)$ is called Cassaigne's spectrum (see  \cite[§4]{cassaigne1999limit}). It is a compact subset of $[0,+\infty]$ with empty interior. Let us describe shortly the ideas behind the computation of the exponents of $\xi_\phi$. The theory of continued fractions (see for example \cite[Chapter I]{schmidt1996diophantine}) ensures that the numerator and denominator of the convergents of $\xi_\phi$ are given by the coefficients of the matrices $\Phi(u)$, where $u$ is a prefix of $w_\phi$ and $\Phi:\AAA^*\rightarrow \MM:=\Mat(\bZ)\cap \GL(\bQ)$ is the morphism of monoids defined by
\begin{align}
\label{eq intro: Phi_A,B fractions continues}
    \Phi(a)= \left(\begin{array}{cc} a & 1\\ 1 & 0 \end{array}\right) \AND \Phi(b)= \left(\begin{array}{cc} b & 1\\ 1 & 0 \end{array}\right).
\end{align}
Moreover, when $u$ is a palindrome, the matrix $\Phi(u)$ is symmetric and the mirror formula provides good simultaneous approximations $(p_j/q_j,p_{j-1}/q_j)$ to $(\xi_\phi,\xi_\phi^2)$ (see \cite{adamczewski2007reversals} and \cite{adamczewski2007palindromic} for other results based on this property). Yet, $w_\phi$ has a lot of palindromic prefixes (see \cite[Lemma 5.3]{bugeaud2005exponentsSturmian}). They yield enough explicit simultaneous approximations to $\xi_\phi$, $\xi_\phi^2$ to compute $\lambda_2(\xi_\phi)$ and $\hlambda_2(\xi_\phi)$. To obtain $\omega_2^*(\xi_\phi)$ and $\homega_2^*(\xi_\phi)$, Bugeaud and Laurent consider the quadratic numbers $\alpha_k:=[0;w_kw_k\cdots]$ for $k\geq 1$ (see \cite[§6]{bugeaud2005exponentsSturmian}). They are very good approximations to $\xi_\phi$ since $w_k^{1+s_{k+1}}$ is a common prefix of $w_\phi$ and $w_kw_k\cdots$ (see \cite[Lemma~5.2]{bugeaud2005exponentsSturmian}). Finally, to get the last pair $\omega_2(\xi_\phi)$, $\homega_2(\xi_\phi)$, they use the polynomials $P_k$ defined below. They notice that $1/\alpha_k$ is the fixed point of the homography (fractional linear transformation) associated to the matrix $\bw_k:= \Phi(w_k)$. In particular, setting
\begin{align}
\label{eq intro: U version polynomial}
    \UU(\bw) := -c + (a-d)X+bX^2 \quad \textrm{for each } \bw = \left(\begin{array}{cc}  a & b \\ c & d \end{array}\right)\in \Mat(\bR),
\end{align}
and defining $P_k:=U(\bw_k)$, we have $P_k(\alpha_k) = 0$, and $P_k(\xi_\phi)$ tends to $0$ very ``quickly''.

\medskip

In a preceding paper \cite{poels2017exponents}, we consider instead a general morphism $\Phi:\AAA^*\rightarrow \MM$ where $A:=\Phi(a)$ and $B:=\Phi(b)$ are any matrices of $\MM$ such that $\det(AB-BA)\neq 0$ and that the content of $\bw_k:=\Phi(w_k)$ is bounded for $k\geq 0$. Surprisingly, it is still possible to built a sequence of symmetric matrices $(\by_i)_{i\geq 0}$ from $(\bw_k)_{k\geq 0}$, which plays the exact same role as the one Laurent and Bugeaud construct from the palindromic prefixes of $w_\phi$ (even though $A$ and $B$ are not necessarily symmetric themselves). Under some technical conditions on the growth of $(\bw_k)_{k\geq 0}$ and $(\det(\bw_k))_{k\geq 0}$, we prove that $(\by_i)_{i\geq 0}$ converges projectively to a symmetric matrix $\left(\begin{array}{cc} 1 & \xi\\ \xi & \xi^2 \end{array}\right)$. In \cite{poels2017exponents} we further give explicit formulas for the four exponents $\lambda_2$, $\hlambda_2$, $\omega_2$ and $\homega_2$ associated to $\xi$ when $\us$ is bounded. They generalize the formulas of Theorem~\ref{thm:bugeaud-laurent,sturm} with the introduction of a second parameter $\delta$ associated to the growth of $|\det(\bw_k)|$ in addition to $\sigma$, see Theorem~\ref{reciproque Thm 3-systeme nombre quasi sturmien intro}. Indeed, compared to the morphism $\Phi$ defined by \eqref{eq intro: Phi_A,B fractions continues}, which yields $\det(\bw_k)=\pm 1$ for each $k$, this new construction provides a larger set of singular points by allowing $|\det(\bw_k)|$ to diverge, as Roy did for the Fibonacci word in \cite{roy2007two}. 

\medskip

In this paper, we complete and extend the results of \cite{poels2017exponents} by considering a general sequence $\us$ (not necessarily bounded) and by relaxing the condition that the content of the matrices $\bw_k=\Phi(w_k)$ is bounded, which brings in the delicate question of controlling it. Moreover, we also compute the exponents $\omega_2^*$ and $\homega_2^*$ that were left out in our previous study. To do this, we use the following surprising phenomenon. For each $k\geq 0$, write $P_k:=\UU(\bw_k)$ and denote by $\alpha_k$ the root of $P_k$ closest to $\xi$. Then $P_k(\xi)$ tends to $0$ as $k$ tends to infinity, and $(\alpha_k)_{k\geq 0}$ is the sequence of best quadratic approximations to $\xi$, exactly as in the continued fraction case (although we are working with two general matrices $A$ and $B$). For each matrix $\bw\in\Mat(\bR)$ (resp. polynomial $P$) we denote by $\norm{\bw}$ (resp. $H(P)$) the largest absolute value of its coefficients.
\begin{definition}
    \label{Def fonction sturmienne, version polynomiale}
    Let $\us = (s_k)_{k\geq 1}$ be a sequence of positive integers and write $\sigma:=\sigma(\us)$ as in \eqref{eq: def sigma}. The set $\Sturm(\us)$ consists of real numbers $\xi$ which are neither rational nor quadratic, such that there exists a sequence of matrices $(\bw_k)_{k\geq 0}$ in $\MM:=\Mat(\bZ)\cap \GL(\bQ)$ with the following properties. For each $k\geq 0$, we denote by $c_k$ the content of $\bw_k$ (defined as the greatest common divisor of its coefficients) and we write $\bwt_k:=c_k^{-1}\bw_k$ and $P_k:=U(\bwt_k)$. Then
    \begin{enumerate}[label=\rm(\roman*)]
        \item $\bw_{k+1}=\bw_k^{s_{k+1}}\bw_{k-1}$ for each $k\geq 1$ and $\det(\bw_0\bw_1-\bw_1\bw_0)\neq 0$.
            \smallskip
        \item There exists $c>0$ such that $\norm{\bw_k^{\ell+1}\bw_{k-1}} \geq c \norm{\bw_k}\norm{\bw_k^\ell\bw_{k-1}}$ for each $k,\ell$ with $k\geq 1$ and $0\leq \ell \leq s_{k+1}$.
            \smallskip
        \item \label{item: Def sturm poly condition 3} $P_k(\xi)/H(P_k)$ tends to $0$ as $k$ tends to infinity.
            \smallskip
        \item \label{item: Def sturm poly condition 4}  $|\det(\bwt_k)| \leq \norm{\bwt_k}^{\sigma/(1+\sigma)+o(1)}$ as $k$ tends to infinity.
    \end{enumerate}
    We define the set $\Sturm$ by
    \[
        \Sturm := \bigcup_{\us}\Sturm(\us).
    \]
\end{definition}

According to our main result below each $\xi\in\Sturm(\us)$ satisfies $\lambda_2(\xi)\geq 1/(1+\sigma) > 1/2$, and thus is transcendental (see \cite[Theorem 2.10]{bugeaud2015exponents}). As we will see $\Sturm(\us)$ is countably infinite (see the remarks after Definition~\ref{def: Sturm(psi)}). It also contains the real numbers $\xi_\phi$ associated to $\us$, and Theorem~\ref{thm:bugeaud-laurent,sturm} follows as a special case of our main result below, taking for granted that $\delta(\xi_\phi)=0$.

\begin{theorem}
\label{reciproque Thm 3-systeme nombre quasi sturmien intro}
Let $\us$ be a sequence of positive integer and set $\sigma:=\sigma(\us)$. There is a function
\[
    \delta:\Sturm(\us)\rightarrow [0,\sigma/(1+\sigma)],
\]
whose image $\Delta(\us)$ is a dense subset of $[0,\sigma/(1+\sigma)]$, with the following property. For each $\xi\in\Sturm(\us)$,
writing $\delta:=\delta(\xi)$, we have
\begin{align*}
    \begin{array}{ll}
       \homega_2(\xi) = \homega_2^*(\xi) = 1+(1-\delta)(1+\sigma),
       & \displaystyle \quad  \hlambda_2(\xi) = \frac{(1-\delta)(1+\sigma)}{1+(1-\delta)(1+\sigma)},\\
        \displaystyle \omega_2(\xi) = \omega_2^*(\xi) = \frac{2-\delta}{\sigma}+1-\delta,
       & \displaystyle\quad 1-\delta\leq\lambda_2(\xi)\leq\max\Big(1-\delta,\frac{1}{1-\delta+\sigma}\Big).
   \end{array}
\end{align*}
If moreover $\delta$ satisfies the stronger condition $\delta < h(\sigma)$, where
$h(\sigma) = \sigma/2+1-\sqrt{(\sigma/2)^2+1}$, then
\[
    \lambda_2(\xi) = 1-\delta\AND \lambdaL(\xi) = \frac{(1-\delta)(1+\sigma)}{2+\sigma}.
\]
The formula for $\lambda_2(\xi)$ still holds if $\delta = h(\sigma)$.
\end{theorem}

The exponent $\lambdaL(\xi)$ above is defined in \cite{poels2019newExpo} and related to work of Fischler \cite{fischler2007palindromic}. We recall its definition in the next section.

\medskip

Recall that the spectrum of an exponent $\nu$ is the set of its values $\Sp(\nu) := \nu(\bR\setminus\overline{\bQ})$. The spectrum of $\omega_2$, resp. $\omega_2^*$, is equal to $[2,+\infty]$ by a result of Bernik \cite{bernik1983use}, resp. Baker and Schmidt \cite{baker1970diophantine}. We also have $\Sp(\lambda_2)=[1/2,+\infty]$ by \cite{beresnevich2007diophantine} and \cite{vaughan2006diophantine}. However, the spectrum of the uniform exponents is more mysterious and complicated. Let $\usFibo=(s_k)_{k\geq 1}$ denote the constant sequence $s_k=1$ for each $k\geq 1$. Its associated Sturmian characteristic word is the Fibonacci word. The set $\Sturm(\usFibo)$ is of particular interest, since it contains Roy's extremal numbers \cite{roy2004approximation} and Fibonacci type numbers \cite{roy2007two}. We have $\sigma(\usFibo) = 1/\gamma$, where $\gamma=(1+\sqrt 5)/2$ is the golden ratio, and the set $\Delta(\usFibo)$ is dense in $[0,1/\gamma^2]$. From this we recover the result of Roy according to which the spectrum of $\hlambda_2$ and that of $\homega_2$ are dense in $[1/2,1/\gamma]$ and in $[2,\gamma^2]$ respectively \cite{roy2007two}. Our first corollary follows by applying Theorem~\ref{reciproque Thm 3-systeme nombre quasi sturmien intro} to the sequence $\usFibo$.

\begin{corollary}
    \label{cor: spec homega_2^* dense}
    The spectrum of $\homega_2^*$ contains a dense subset of the interval $[2,\gamma^2]$.
\end{corollary}

Theorem~\ref{thm:bugeaud-laurent,sturm} gives $2+\sigma(\us)\in \Sp(\homega_2^*)$ for each $\us$. Nonetheless, the set of values taken by $2+\sigma(\us)$ is a compact subset of $[2,\gamma^2]$ with empty interior, and thus far from being dense (see the paper of Cassaigne \cite{cassaigne1999limit}). Also note that a theorem of Bugeaud (see \cite{bugeaud2010simultaneousReal} and \cite[Theorem 5.6]{bugeaud2015exponents}) shows that the full interval $[1,3/2]$ is contained in the spectrum of $\homega_2^*$.

\medskip

Although it is possible to have $\omega_2^*(\xi) < \omega_2(\xi)$ (see \cite[Theorem 5.7]{bugeaud2015exponents} and \cite{bugeaud2012continued} for explicit examples), our next corollary, proven in Section~\ref{subsection: proofs des thm}, shows that it does not happen if $\homega_2(\xi)$ is sufficiently close to its maximal value $\gamma^2$.

\begin{corollary}
    \label{cor: spec homega_2^* dense II}
    There exists $\ee > 0$ with the following property. Let $\xi$ be a real number which is neither rational nor quadratic. If $\homega_2(\xi) > \gamma^2-\ee$, then $\xi\in\Sturm(\usFibo)$. In particular
    \[
        \homega_2^*(\xi)=\homega_2(\xi)\AND \omega_2^*(\xi)=\omega_2(\xi),
    \]
    and the set $\Sp(\homega_2^*)\cap[\gamma^2-\ee,\gamma^2]$ is countably infinite.
\end{corollary}

Applying Theorem~\ref{reciproque Thm 3-systeme nombre quasi sturmien intro} to the sequence $\usFibo$ we also deduce new information on $\Sp(\lambdaL)$.

\begin{corollary}
    The spectrum of $\lambdaL$ contains a dense subset of the interval $[\kappa,1/\gamma]$, where
    \[
        \kappa := \frac{1-h(1/\gamma)}{\gamma} = 0.4558\cdots
    \]
\end{corollary}

In this paper, we will use an equivalent definition of the set $\Sturm(\us)$ (see Section~\ref{subsection: proofs des thm}), which connects to a class of numbers considered by Fischler in \cite{fischler2007palindromic}. Before stating it, let us go back to the combinatorial properties of the Sturmian characteristic word $w_\phi$. If $u,v,w\in\AAA^*$ satisfy $u=vw$, we write $v^{-1}u = w$. If $u$ has length at least $2$, then we denote by $u'$ the word $u$ deprived of its two last letters. By \cite[Lemma 5.3]{bugeaud2005exponentsSturmian}, the sequence $(\pi_i)_{i\geq 0}$ of palindromic prefixes of $m_\phi$ (ordered by increasing length) consists of $b,\dots,b^{s_1-1}$ and the words
\begin{align*}
    (w_k^{\ell+1}w_{k-1})' \quad \textrm{with $k\geq 1$ and $0\leq \ell < s_{k+1}$.}
\end{align*}
Moreover, there exists a function $\psi:=\psi_{\us}$ defined over $\bN$ (with $\psi(i) < i$ for each $i\in\bN$) such that
\begin{align}
    \label{eq intro: rec functio psi pour mots}
    \pi_{i+1} = \pi_i (\pi_{\psi(i)}^{-1} \pi_i)
\end{align}
for each large enough $i$ (see \cite[§3]{fischler2006palindromic}). The precise definition of $\psi_\us$ is given in the next section (see Definition~\ref{Def fonction sturmienne}).
In the Fibonacci case where $\us=\usFibo$, the associated function $\psi$ satisfies $\psi(i)=i-2$ for each $i$. In \cite{fischler2007palindromic}, Fischler studied real numbers $\xi_w$ associated to words $w$ with a large density of palindromic prefixes (also see \cite{fischler2006palindromic}). He introduced a new Diophantine exponent $\beta_0$ (whose definition is recalled in the next section) which is closely related to $\hlambda_2$, and was able to compute $\beta_0(\xi_w)$ and to give a complete description of the set $\beta_0(\bR\setminus\overline{\bQ}) \cap (1,2)$. Our primary motivation to introduce the new class of numbers $\Sturm(\us)$ comes from the following result (a combination of \cite[Theorem~4.1]{fischler2007palindromic} with \cite[Lemma~7.1]{fischler2006palindromic}), where we identify $\bR^3$ with the space of matrices $\Mat(\bR)$ under the map
\begin{equation}
\label{eq: identification R^3 avec Mat(R)}
    (x_0,x_1,x_2) \longmapsto \left( \begin{array}{cc} x_0 & x_1 \\ x_1 & x_2\end{array}\right),
\end{equation}
and we denote by $\Adj(\bw)$ the adjoint of a matrix $\bw\in\Mat(\bR)$.

\begin{theorem}[Fischler, 2007]
\label{Rappels Thm 4.1 SF VO}
Let $\xi$ be a real number with $\beta_0(\xi) < 2$, which is neither rational nor quadratic. Then, there exists a sequence $(\bv_i)_{i\geq 0}$ of non-zero primitive points in $\bZ^3$ (identified with the corresponding symmetric matrices) with the following properties. The sequence $(\norm{\bv_i})_{i\geq 0}$ tends to infinity,
\begin{equation}
    \label{eq: thm fischler}
   \norm{\bv_i\wedge\Xi} = \norm{\bv_i}^{-1+o(1)},
\end{equation}
where $\Xi:=(1,\xi,\xi^2)$, and there exists a function $\psi:\bN\rightarrow\bN$ such that $\bv_{i+1}$ is collinear to $\bv_i\Adj(\bv_{\psi(i)})\bv_i$ for each large enough $i$. If moreover $\beta_0(\xi) < \sqrt 3$, then we may choose $\psi=\psi_\us$ for a bounded sequence $\us$ of positive integers.
\end{theorem}

The result of Fischler is more precise. It shows that $\psi$ belongs to a narrow class of functions called \textsl{asymptotically reduced} (see \cite[Definition 2.1]{fischler2007palindromic}), which includes all functions  $\psi_\us$ with $\us$ bounded.
In general each asymptotically reduced function comes from an infinite word with a large density of palindromic prefixes $\pi_i$ in such a way that the recurrence \eqref{eq intro: rec functio psi pour mots} holds \cite[Section~3.1]{fischler2006palindromic}. In \cite{fischler2007palindromic}, Fischler motivates and asks the following question.

\bigskip

\noindent\textbf{Problem.} Let $\xi\in\bR$ which is neither rational nor quadratic. Does the condition $\beta_0(\xi)<2$ imply $\hlambda_2(\xi)=1/\beta_0(\xi)$ ?

\bigskip

The following partial answer proves a claim made by Fischler in \cite{fischler2004spectres}.


\begin{theorem}
\label{Thm beta_0 < sqrt}
    Let $\xi$ be a real number which is neither rational nor quadratic. If $\beta_0(\xi)<\sqrt 3$, then $\hlambda_2(\xi)  = \lambdaL(\xi) = 1/\beta_0(\xi)$.
\end{theorem}

The idea is to prove that the condition $\beta(\xi)< \sqrt 3$ implies that $\xi\in \Sturm(\us)$ and $\delta(\xi)=0$, where  $\us$ is the sequence given by Theorem~\ref{Rappels Thm 4.1 SF VO}, and $\delta(\xi)$ is the quantity appearing in Theorem~\ref{reciproque Thm 3-systeme nombre quasi sturmien intro} (see Section~\ref{subsection: proofs des thm}). The first observation follows relatively easily from the following alternative definition of the set $\Sturm(\us)$ (see Section~\ref{subsection: proofs des thm}).

\begin{definition}
    \label{def: Sturm(psi)}
    Let $\us$ be a sequence of positive integers, write $\sigma:=\sigma(\us)$ and  $\psi=\psi_\us$. The set $\Sturm(\us)$ is the set of real numbers $\xi$ which are neither rational nor quadratic, and for which there exists a sequence $(\by_i)_{i\geq 0}$ of non-zero primitive points in $\bZ^3$ (identified with their symmetric matrices) with the following properties.
    \begin{enumerate}[label=\rm(\roman*)]
        \item \label{item:def Sturm(psi):item 1} The sequence $(\by_i)_{i\geq 0}$ converges projectively to $\Xi:=(1,\xi,\xi^2)$.
            \smallskip
        \item \label{item:def Sturm(psi):item 2} The matrix $\by_{i+1}$ is  proportional to $\by_i\Adj(\by_{\psi(i)})\by_i$ for each large enough $i$.
            \smallskip
        \item \label{item:def Sturm(psi):item 3} We have $|\det(\by_i)| \leq \norm{\by_i}^{\sigma/(1+\sigma)+o(1)}$ as $i$ tends to infinity with $\psi(i+1) < i$.
    \end{enumerate}
\end{definition}

 We end this introduction with a few remarks, see Section~\ref{subsection: new characterization of Sturm} for more details.

\begin{remarks}
Since $\xi\notin\bQ$, the sequence $(\norm{\by_i})_{i\geq 0}$ tends to infinity and $\det(\by_i)\neq 0$ for $i$ large enough (since for large $i$, by \ref{item:def Sturm(psi):item 2} and \ref{item:def Sturm(psi):item 1}, the kernel of the matrix $\by_i$ is included in the kernel of $\Xi$)
\smallskip

The parameter $\delta(\xi)$ in Theorem~\ref{reciproque Thm 3-systeme nombre quasi sturmien intro} can be defined as the supremum limit of $\log|\det(\by_i)|/\log\norm{\by_i}$ as $i$ tends to infinity with $\psi(i+1) < i$.
\smallskip

In view of \eqref{eq intro: rec functio psi pour mots}, if $\xi=\xi_\phi$, then we can take $\by_i=\Phi(\pi_i)$ for $i$ large, where $\Phi$ is as in \eqref{eq intro: Phi_A,B fractions continues}. Since in that case $\det(\by_i)=\pm 1$, we have $\delta(\xi_\phi)=0$.
\smallskip

By \ref{item:def Sturm(psi):item 2}, the sequence $(\by_i)_{i\geq i_0}$ (with $i_0$ large enough) is entirely determined by three points in $\bZ^3$. We thus have a surjection $(\bZ^3)^3\rightarrow\Sturm(\us)$, and $\Sturm(\us)$ is therefore at most countable.
\smallskip

It follows easily from \cite[Definition 6.2]{poels2017exponents} that the set of Sturmian type numbers constructed in \cite{poels2017exponents} is included in $\Sturm$. More precisely, if $\us$ is bounded, then the (infinite) set of proper $\psi_\us$-Sturmian numbers (see Definition~\ref{def:sturmian type numbers} and \cite[Proposition 6.1]{poels2017exponents}) is included in $\Sturm(\us)$. This yields the density of $\Delta(\us)$ in $[0,\sigma/(1+\sigma)]$ when $\sigma=\sigma(\us) > 0$.
\end{remarks}

Our paper is organized as follows. In the next section, we define the Diophantine exponents involved in Theorem~\ref{reciproque Thm 3-systeme nombre quasi sturmien intro} and introduce some notation. Sections~\ref{section: combinatorics of Sturmian seq} and \ref{section: estimates for Sturmian sequences} are devoted to the theory of Sturmian sequences of matrices; in the former we focus on the combinatorial aspects and establish new key-identities, while in the latter we study the asymptotic behavior of these sequences. Combining those results, we obtain a new characterization of the set $\Sturm$ in terms of Sturmian sequences of matrices (see
Section~\ref{subsection: new characterization of Sturm}). This allows us, using parametric geometry of numbers, to prove our main theorem in the last section.\\\\

\section{Notation}
\label{section: notation}

Given a positive integer $n$ and $\bx\in\bR^n$, we define its norm $\norm{\bx}$ as the largest absolute value of its coordinates. Let $\xi$ be a real number which is neither rational nor quadratic. We associate to $\xi$ several classical Diophantine exponents as follows. The ordinary (resp. uniform) exponent of simultaneous approximation $\lambda_2(\xi)$, (resp. $\hlambda_2(\xi)$) is the supremum of real numbers $\lambda$ such that, for arbitrarily large values of $X$ (resp. for each $X$ large enough), there exists $\bx\in\bZ^3\setminus\{0\}$ satisfying
\[
    \norm{\bx\wedge\Xi} \leq X^{-\lambda} \AND \norm{\bx} \leq X,
\]
where $\bx\wedge\by\in\bR^3$ denotes the cross product of $\bx$ and $\by$ in $\bR^3$. Similarly, the ordinary (resp. uniform) exponent $\omega_2(\xi)$, (resp. $\homega_2(\xi)$) is the supremum of real numbers $\omega$ such that, for arbitrarily large values of $X$ (resp. for each $X$ large enough), there exists $\bx\in\bZ^3\setminus\{0\}$ satisfying
\[
    |\bx\cdot\Xi|\leq X^{-\omega} \AND \norm{\bx} \leq X,
\]
where $\bx\cdot\by$ denotes the standard scalar product of $\bx$ and $\by$ in $\bR^3$. The exponent $\omega_2^*(\xi)$ (resp. $\homega_2^*(\xi)$) is the supremum of real numbers $\omega$ such that, for arbitrarily large values of $X$ (resp. for each $X$ large enough), there is an algebraic number $\alpha$ of degree at most $2$ satisfying
\begin{align*}
    0 < |\xi-\alpha|\leq H(\alpha)^{-1}X^{-\omega} \AND H(\alpha) \leq X,
\end{align*}
where $H(\alpha)$ is the \textsl{height} of $\alpha$, defined as the largest absolute value of the coefficients of its irreducible minimal polynomial over $\bZ$. See \cite{bugeaud2005exponentsSturmian} for the motivation of the division by $H(\alpha)$ in the left-hand side. We now recall the definitions of the last two exponents $\beta_0(\xi)$ and $\lambdaL(\xi)$, introduced respectively by Fischler in \cite{fischler2007palindromic} and by the author in \cite{poels2019newExpo} on the basis of \cite{fischler2007palindromic}. Set $\Xi=(1,\xi,\xi^2)$ and for each $0\leq \mu < \lambda(\xi)$, denote by $\hlambda_\mu(\Xi)$ the supremum of the real numbers $\lambda$ for which
\begin{align*}
     \norm{\bx\wedge\Xi} \leq \min(X^{-\lambda},\norm{\bx}^{-\mu}) \AND \norm{\bx} \leq X
\end{align*}
admits a non-zero integer solution for each sufficiently large value of $X$. The map $\mu\mapsto\hlambda_\mu(\Xi)$ is non-increasing, and for $\mu=0$ we simply have $\hlambda_0(\Xi)=\hlambda_2(\xi)$. We set
\[
    \beta_0(\xi) = \left\{ \begin{array}{ll}
        \lim_{\mu\rightarrow 1^-}1/\hlambda_\mu(\Xi) & \textrm{if $\lambda_2(\xi)\geq 1$,} \\
        \infty & \textrm{else}
    \end{array}\right.,
    \AND
    \lambdaL(\xi) = \lim_{\mu\rightarrow \lambda_2(\xi)^-}\hlambda_\mu(\Xi).
\]
In particular $\lambdaL(\xi) \leq \hlambda_2(\xi)$. Note that the definition of $\lambdaL$ in \cite{poels2019newExpo} applies to general points $\Xi=(1,\xi,\eta) \in \bR^3$. In the current situation, the two above exponents are connected in the following way: if $\beta_0(\xi) < 2$, then $\lambda_2(\xi)=1$ and $\lambdaL(\xi)=1/\beta_0(\xi)$ (see \cite[Lemma~1.3]{poels2019newExpo}). The classical general estimates below are valid for each $\xi\in\bR$ which is neither rational nor quadratic. Recall that $\gamma = (1+\sqrt 5)/2$ denotes the golden ratio. First
\begin{align*}
    \frac{1}{2} \leq \hlambda_2(\xi) \leq 1/\gamma \AND 2 \leq \homega_2(\xi) \leq \gamma^2.
\end{align*}
The lower bounds are obtained by the Dirichlet box principle, the upper bounds follow respectively from  \cite[Theorem $1a$]{davenport1969approximation} and from \cite{arbourRoyCriterionDegreTwo}. Jarn\'ik's identity \cite[Theorem 1]{jarnik1938khintchineschen} links $\hlambda_2(\xi)$ and $\homega_2(\xi)$ as follows
\begin{equation}
\label{Eq Jarnik}
    \hlambda_2(\xi) = 1 - \frac{1}{\homega_2(\xi)}.
\end{equation}
We also have (see \cite[Theorem 2.5]{bugeaud2015exponents})
\begin{align}
    \label{eq: intro general pour exposants polynomials}
    1 \leq \homega_2^*(\xi) \leq \min\{\omega_2^*(\xi),\homega_2(\xi)\} \leq \max\{\omega_2^*(\xi),\homega_2(\xi)\} \leq \omega_2(\xi).
\end{align}

We now recall the notion of \textsl{Sturmian functions} $\psi_\us$, which intervene in the recurrence relation \eqref{eq intro: rec functio psi pour mots} of the palindromic prefixes of a Sturmian characteristic word. They play a central role in \cite{poels2017exponents} (see also \cite{fischler2007palindromic} and \cite{fischler2006palindromic}).

\begin{definition}
    \label{Def fonction sturmienne}
    Let $\us = (s_k)_{k\geq 1}$ be a sequence of positive integers and for each $k\geq 0$ set $t_k = s_0+s_1+\dots+s_k$ (where $s_0 = -1$). We associate to $\us$ a function $\psi=\psi_\us$ defined on $\bN$ as follows.
    \begin{equation*}
        \psi(i) :=
        \left\{ \begin{array}{ll}
            t_{k-1} - 1 & \textrm{if $i=t_k$ with $k\geq 1$},\\
            i-1 & \textrm{else}. 
        \end{array} \right.
    \end{equation*}
    Note that $\us$ is entirely characterised by $\psi$, the sequence $(t_k)_{k\geq 1}$ consisting of the integers $n$ such that $\psi(n)\leq n-2$.
\end{definition}

We denote by $\norm{\bw}$  the norm of a matrix $\bw\in\Mat(\bR)$ defined as the largest absolute value of its coefficients. Recall that $\bR^3$  is identified with $\Mat(\bR)$ under the map \eqref{eq: identification R^3 avec Mat(R)}. Accordingly, we define the determinant $\det(\bx)=x_0x_2-x_1^2$ of a point $\bx=(x_0,x_1,x_2)\in\bR^3$. Similarly, given symmetric matrices $\bx,\by$, we write $\bx\wedge\by$ to denote the cross product of the corresponding points in $\bR^3$. We also identify $\bR^3$ (and thus $\Mat(\bR)$) to $\bR[X]_{\leq 2}$, the space of polynomial of degree at most $2$, via the map $(x_0,x_1,x_2) \longmapsto x_0+x_1X+x_2X$.\\
For any $\bw\in\Mat(\bR)$, we denote by $\transpose{\bw}$ its transpose, and by $\Adj(\bw)$ its adjoint. The \textsl{content} of a non-zero matrix $\bw\in\Mat(\bZ)$ or of a non-zero point $\by\in\bZ^3$ is the greatest common divisor of its coefficients. We say that such a matrix or point is \textsl{primitive} if its content is $1$. More generally, if $\bw\in\Mat(\bR)\setminus\{0\}$ is proportional to a matrix of $\Mat(\bZ)$, we say that $\bw$ is \textsl{defined over $\bQ$} and we denote by $\cont(\bw)$ the positive real number $\alpha$ such that $\alpha^{-1}\bw$ is a primitive matrix of $\Mat(\bZ)$. We set
\begin{equation}
\label{Eq def matrice J}
    J = \left(\begin{array}{cc} 0 & 1\\ -1 & 0 \end{array}\right) \AND \Id = \left(\begin{array}{cc} 1 & 0\\ 0 & 1 \end{array}\right).
\end{equation}
Given a non-empty interval $A\subset\bR$, a positive integer $n$ and $F:A\rightarrow \bR^n$, we denote by $\normI{F}:= \sup_{q\in A} \norm{F(q)}$. Finally, let $I$ be a set (typically of the form $\bN^r$), $(a_{\underline{i}})_{\underline{i}\in I}$ and let $(b_{\underline{i}})_{\underline{i}\in I}$ be two sequences of non-negative real numbers indexed by $I$. For any non-empty subset $J\subseteq I$, we write \say{$a_{\underline{i}} \ll b_{\underline{i}}$ for $\underline{i} \in J$} or \say{$b_{\underline{i}} \gg a_{\underline{i}}$ for $\underline{i} \in J$} if there is a constant $c>0$ such that for each $\underline{i}\in J$ we have $a_{\underline{i}} \leq cb_{\underline{i}}$. We write \say{$a_{\underline{i}}\asymp b_{\underline{i}}$ for $\underline{i} \in J$} if both $a_{\underline{i}}\ll b_{\underline{i}}$ and $b_{\underline{i}}\ll a_{\underline{i}}$ for $\underline{i} \in J$ hold. In the special case where $I = \bN$, unless otherwise stated, we will always implicitly take $J$ of the form $[j_0,+\infty) \cap \bN$ for $j_0$ large enough, and we will simply write $a_{\underline{i}} \ll b_{\underline{i}}$, $b_{\underline{i}} \gg a_{\underline{i}}$ and $a_{\underline{i}} \asymp b_{\underline{i}}$.

\section{Combinatorics of Sturmian sequences of matrices}
\label{section: combinatorics of Sturmian seq}

Let $\us = (s_k)_{k\geq 1}$ be a sequence of positive integers (not necessarily bounded) and set $\psi=\psi_\us$ (see Definition~\ref{Def fonction sturmienne}). We define below the notions of $\psi$-Sturmian sequences and admissible $\psi$-Sturmian sequences of matrices. We develop the latter notion in §\ref{subsection: admissibility prop}. To an admissible $\psi$-Sturmian sequence correspond two sequences of symmetric matrices $(\by_i)_i$ and $(\bz_i)_i$, which we also view as sequences in $\bR^3$ (see §\ref{subsection: symmetric matrices associated to}). In our applications, $(\by_i)_i$ will provide ``good'' solutions to the problem of simultaneous approximation, whereas $(\bz_i)_i$ will be related to the problem with polynomials. In §\ref{subsection: fonction U}, we establish a new and surprising formula for $(\bz_i)_i$. This is one of the key-properties for studying the exponents $\omega_2^*$ and $\homega_2^*$.

\begin{definition}
A $\psi$\textsl{-Sturmian} sequence in $\GL(\bR)$ is a sequence $(\bw_k)_{k\geq 0}$ such that $\bw_0,\bw_1\in\GL(\bR)$, and for each $k\geq 1$, we have the recurrence relation $\bw_{k+1} = \bw_k^{s_{k+1}}\bw_{k-1}$.
\end{definition}

Clearly, such a sequence is entirely determined by its first two elements $\bw_0$ and $\bw_1$.

\begin{definition}
    Let  $(\bw_k)_{k\geq 0}$ be a $\psi$-Sturmian sequence in $\GL(\bR)$.
\begin{itemize}
    \item We say that $(\bw_k)_{k\geq 0}$ is \textsl{admissible} if the matrix $\bw_0\bw_1-\bw_1\bw_0$ is inversible. Given an integer $k\geq 1$, the identity $\bw_k\bw_{k+1}-\bw_{k+1}\bw_k = -\bw_k^{s_{k+1}-1}(\bw_{k-1}\bw_k-\bw_k\bw_{k-1})$ implies that $(\bw_k)_{k\geq 0}$ is admissible if and only if $\bw_{k-1}\bw_k-\bw_k\bw_{k-1}$ is inversible.
    \item The sequence $(\bw_k)_{k\geq 0}$ has a \textsl{multiplicative growth} if $\norm{\bw_k^{\ell+1}\bw_{k-1}} \asymp \norm{\bw_k}\norm{\bw_k^{\ell}\bw_{k-1}}$ for $k\geq 1$ and $0\leq \ell \leq s_{k+1}$
    \item Finally, $(\bw_k)_{k\geq0 }$ is \textsl{defined over $\bQ$} if, for each $k \geq 0$, the matrix $\bw_k$ is proportional to a matrix of $\Mat(\bQ)$.
\end{itemize}
\end{definition}

\subsection{Admissible sequences} ~ \medskip
\label{subsection: admissibility prop}

In \cite{roy2007two} and \cite{poels2017exponents} a $\psi$-Sturmian sequence $(\bw_i)_{i\geq 0}$ is said to be admissible if there exists a matrix $N\in\GL(\bR)$ satisfying
\begin{equation}
\label{eq: condition symetrie admissible}
    \bw_0\transpose{N},\quad\bw_1N,\AND \bw_1\bw_0\transpose{N}\quad \textrm{are symmetric},
\end{equation}
which gives a slightly different notion than ours. Note that by taking the transpose of $\bw_1\bw_0\transpose{N}$ and using successively the fact that $\bw_0\transpose{N}$ and $\bw_1N$ are symmetric, \eqref{eq: condition symetrie admissible} implies that
\begin{align}
    \label{eq: lien w0w1 w1w0}
    \bw_1\bw_0\transpose{N} = \bw_0\bw_1N.
\end{align}
In that case and if $N$ is symmetric, then $\bw_0$ and $\bw_1$ commute; this is a degenerate situation that we want to avoid. According to the next result, if $\bw_0$ and $\bw_1$ do not commute, then our definition of admissibility is equivalent to the existence of $N\in\GL(\bR)$ satisfying \eqref{eq: condition symetrie admissible}. In addition, up to a multiplicative constant, we provide a simple expression for $N$.

\begin{proposition}
\label{prop: conditions equivalentes admissibilite}
Let $\bw_0,\bw_1\in\GL(\bR)$. Then the following conditions are equivalent:
\begin{enumerate}[label=\rm(\roman*)]
 \item \label{enumi: admissible condition 1} $\det(\bw_0\bw_1-\bw_1\bw_0) \neq 0$
 \smallskip
 \item \label{enumi: admissible condition 2} $\bw_1\bw_0\neq \bw_0\bw_1$ and there exists a matrix $N\in\GL(\bR)$
 satisfying \eqref{eq: condition symetrie admissible}.
\end{enumerate}
If these conditions are satisfied, then any $N\in\Mat(\bR)$ satisfying \eqref{eq: condition symetrie admissible} is proportional to
\begin{equation}
    \label{eq: def N}
    M=(\Id-\bw_1^{-1}\bw_0^{-1}\bw_1\bw_0)J.
\end{equation}
\end{proposition}

\begin{proof}
Note that the matrix $M$ defined in \eqref{eq: def N} is inversible if and only if  \ref{enumi: admissible condition 1} holds. Suppose \ref{enumi: admissible condition 1}. Since $J\bx J = -\det(\bx)\transpose{\bx}^{-1}$ for each $\bx\in\GL(\bR)$, we have
\begin{equation}
    \label{eq fnct U pour 2nd expr wedge}
    \transpose{\big((\Id-\bx^{-1}\by^{-1}\bx\by)J\big)} = -(\Id-\by^{-1}\bx^{-1}\by\bx)J
\end{equation}
for any $\bx,\by\in\GL(\bR)$. In particular $\transpose{M} = (\Id-\bw_0^{-1}\bw_1^{-1}\bw_0\bw_1)J$. Moreover, since $A\in\Mat(\bR)$ is symmetric if and only if $\Tr(AJ) = 0$, using \eqref{eq: def N} and the above expression of $\transpose{M}$, it is easily seen that \eqref{eq: condition symetrie admissible} is satisfied with $N=M$, hence \ref{enumi: admissible condition 2}.

\smallskip

Now we prove \ref{enumi: admissible condition 2} $\Rightarrow$ \ref{enumi: admissible condition 1}. First, note that if $N\in\GL(\bR)$ satisfies \eqref{eq: condition symetrie admissible}, then we have \eqref{eq: lien w0w1 w1w0}. Since by hypothesis $\bw_1\bw_0 \neq \bw_0\bw_1$, it implies that $N$ is not symmetric. Thus
$\det(\transpose{N}- N) \neq 0$ and the matrix $(\bw_0\bw_1-\bw_1\bw_0)\transpose{N} = \bw_0\bw_1(\transpose{N}- N)$ is inversible, hence \ref{enumi: admissible condition 1}.

\smallskip

Suppose now that  \ref{enumi: admissible condition 1} and \ref{enumi: admissible condition 2} are satisfied and let us prove the last part of the proposition. In general, the conditions \eqref{eq: condition symetrie admissible} represent a system of three linear equations in the four unknown coefficients of $N$. By \eqref{eq: lien w0w1 w1w0}, the condition  \ref{enumi: admissible condition 1} implies that there is no non-zero symmetric matrix $N\in\Mat(\bR)$ solution of this system. Its rank is thus equal to $3$ and the space of solution has dimension $1$.
\end{proof}

\subsection{Symmetric matrices associated to Sturmian sequences} ~ \medskip
\label{subsection: symmetric matrices associated to}

We associate to any admissible $\psi$-Sturmian sequence two sequences of symmetric matrices $(\by_i)_{i\geq -2}$
and $(\bz_i)_{i\geq -1}$ as in \cite[Definitions~3.5 and~4.2]{poels2017exponents}. They play a major role in our study.

\begin{definition}
\label{Def (y_i)}
\label{Def (z_i)_i}
Let $(\bw_k)_{k\geq 0}$ be a $\psi$-Sturmian sequence in $\GL(\bR)$, and let $N\in\GL(\bR)$ be such that
$(\by_{-2},\by_{-1},\by_{0}):=(\bw_0\transpose{N},\bw_1N,\bw_1\bw_0\transpose{N})$ is a triple of symmetric
matrices. We define $\bw_{-1} = \bw_0^{-1}\bw_1$, and for each integers $k,\ell \geq 0$ with $0\leq \ell < s_{k+1}$, we set
\begin{align}
\label{Eq Def y_i}
    \by_{t_k + \ell} = \bw_k^{\ell+1}\bw_{k-1}N_k \AND \bz_{t_k+\ell} = \frac{1}{\det(\bw_k)}\by_{\psi(t_{k+1})}\wedge \by_{t_k+\ell},
\end{align}
where $N_k = N$ if $k$ is even, $N_k = \transpose{N}$ if $k$ is odd. By \cite[Proposition~3.6]{poels2017exponents}
the matrix $\by_i$ is symmetric for each $i\geq -2$, so that the wedge product defining $\bz_{t_k+\ell}$ makes sense.
Note that the left-hand side of \eqref{Eq Def y_i} remains valid for $\ell=s_{k+1}$. In particular, for each $k\geq 1$, we have
\begin{equation}
\label{Eq y_psi(k) = w_k-1}
    \by_{\psi(t_k)} = \bw_{k-1}N_k.
\end{equation}
\end{definition}

\begin{remark}
According to Proposition~\ref{prop: conditions equivalentes admissibilite}, if $(\bw_k)_{k\geq 0}$ is admissible,
then the matrix $N$ is proportional to $(\Id-\bw_1^{-1}\bw_0^{-1}\bw_1\bw_0)J$. In the following, if we refer to the
sequences $(\by_i)_{i\geq -2}$ and $(\bz_i)_{i\geq -1}$ associated to an admissible $\psi$-Sturmian sequence without
further precision on $N$, we will always implicitly take $N=(\Id-\bw_1^{-1}\bw_0^{-1}\bw_1\bw_0)J$ in \eqref{Eq Def y_i}.
\end{remark}

Those two sequences satisfy a lot of combinatorial properties, for example \cite[Eq.~(3.4)]{poels2017exponents} yields:
\begin{align}
\label{reciproque recurrence crochet def y_i}
    \by_{i+1} = \by_i\by_{\psi(i)}^{-1}\by_i\quad (i\geq 0).
\end{align}

In the next lemma, we study the degenerate situation where $N$ is symmetric (this is one of the reason why we want to avoid this situation).

\begin{lemma}
\label{lem: degenerate N symmetric}
    Let $(\bw_k)_{k\geq 0}$, $N\in\GL(\bR)$ and $(\by_i)_{i\geq -2}$ be as in  Definition~\ref{Def (y_i)}.
    Then, for each $i\geq -1$ which is not among the $t_k$ ($k\geq 0$), the poins
    $\by_{i-1},\by_i,\by_{i+1}$ are linearly dependent. Moreover, the following assertions are equivalent:
    \begin{enumerate}[label=\rm(\roman*)]
    \item \label{item:lem: degenerate N symmetric: item 1} $(\bw_k)_{k\geq 0}$ is not admissible;
    \smallskip
    \item \label{item:lem: degenerate N symmetric: item 2} $N$ is symmetric;
    \smallskip
    \item \label{item:lem: degenerate N symmetric: item 3} There is $k\geq 0$ such that $\by_{t_k-1},\by_{t_k},\by_{t_k+1}$ are linearly dependent;
    \smallskip
    \item \label{item:lem: degenerate N symmetric: item 4} The space generated by  $(\by_i)_{i\geq -2}$ has dimension at most $2$;
\end{enumerate}
\end{lemma}

\begin{proof}
    Eq. $(2.1)$ of \cite{roy2004approximation} combine with $J\by J = -\det(\by)\by^{-1}$ (valid for all symmetric matrix $\by\in \GL(\bR)$) gives the identity $\det(\bx,\by,\bz) = -\det(\by)\Tr(J\bx \by^{-1}\bz)$ for all symmetric matrices $\bx,\by,\bz\in\GL(\bR)$ (also viewed as points in $\bR^3$). Since $\Tr(JA) = 0$ if and only if $A$ is symmetric, we obtain the following useful criterion, valid for each symmetric matrices $\bx,\by,\bz \in\GL(\bR)$:
    \begin{align}
    \label{{eq proof:lem: degenerate N symmetric}}
        \det(\bx,\by,\bz) = 0 \Leftrightarrow \textrm{$\bx \by^{-1}\bz$ is symmetric.}
    \end{align}
    Now, let $i\geq -1$ be an index not among the $t_k$. Then $\psi(i)=i-1$, and by \eqref{reciproque recurrence crochet def y_i} the matrix $\by_i\by_{i+1}^{-1}\by_{i-1} = \by_{i-1}\by_i^{-1}\by_{i-1}$ is symmetric. We deduce from \eqref{{eq proof:lem: degenerate N symmetric}} that $\det(\by_i,\by_{i+1},\by_{i-1})=0$. This proves the first part of our lemma.

    \ref{item:lem: degenerate N symmetric: item 1} $\Leftrightarrow$ \ref{item:lem: degenerate N symmetric: item 2} by \eqref{eq: lien w0w1 w1w0} and Proposition~\ref{prop: conditions equivalentes admissibilite}. We obtain \ref{item:lem: degenerate N symmetric: item 2} $\Leftrightarrow$ \ref{item:lem: degenerate N symmetric: item 3} by noticing that if $i=t_k$ with $k\geq 0$, then $\by_{i+1} = \bw_k\by_i$, $\by_{i-1} = \by_{\psi(t_{k+1})} = \bw_kN_{k+1}$, so that $\by_i\by_{i+1}^{-1}\by_{i-1} = N_{k+1}$, and $\det(\by_{t_k},\by_{t_k+1},\by_{t_k-1})=0$ if and only if $N$ is symmetric. Lastly, we get \ref{item:lem: degenerate N symmetric: item 3} $\Leftrightarrow$ \ref{item:lem: degenerate N symmetric: item 4} by combining the above combined with the first part of the lemma.
\end{proof}

We now prove that any sequence satisfying \eqref{reciproque recurrence crochet def y_i} comes from a $\psi$-Sturmian sequence.
This will play a crucial role in establishing the new characterization of the set $\Sturm$ in
Section~\ref{subsection: new characterization of Sturm}.

\begin{proposition}
\label{prop: reconstruction suite sturm}
Let $i_0\geq -2$ be an integer and let $(\bv_i)_{i\geq -2}$ be a sequence of symmetric matrices such that
$\det(\bv_i) \neq 0 $ for each $i\geq -2$, and
\begin{equation}
\label{y_psi == à [y_i,y_i,y_i+1]}
    \bv_{i+1} = \bv_i\bv_{\psi(i)}^{-1}\bv_i
\end{equation}
for each $i\geq 0$. Then, there are $N\in\GL(\bR)$ and a $\psi$-Sturmian sequence $(\bw_k)_{k\geq 0}$
with the following property. We have $(\bv_{-2},\bv_{-1},\bv_{0}):=(\bw_0\transpose{N},\bw_1N,\bw_1\bw_0\transpose{N})$,
and $(\bv_i)_{i\geq -2}$ is precisely the sequence $(\by_i)_{i\geq -2}$ associated to $(\bw_k)_{k\geq 0}$ and $N$
by Definition~\ref{Def (y_i)}. If moreover the space generated by $(\bv_i)_{i\geq -2}$ has dimension $3$,
then the sequence $(\bw_k)_{k\geq 0}$ is admissible.
\end{proposition}

The last part of the proposition is implied by Lemma~\ref{lem: degenerate N symmetric}. The first part comes
from Proposition~\ref{prop: reconstruction w_k admissible} below.

\begin{proposition}
\label{Reciproque Propriete sur y_k}
Let $(\bv_i)_{i\geq -2}$ be as in Proposition~\ref{prop: reconstruction suite sturm}, and for each $k\geq 0$, set
$\bw_k := \bv_{t_k+1}\bv_{t_k}^{-1}$. Then, we have the following properties:
\begin{enumerate}[label=\rm(\roman*)]
\item $\bv_{t_k+\ell} = \bw_k^{\ell} \bv_{t_k} = \bw_k^{\ell+1} \bv_{\psi(t_k)}$ for $k\geq 1$ and $0\leq \ell \leq s_{k+1}$.
\label{Reciproque Propriete 1 sur y_k}
\smallskip
\item $\bv_{j+1} = \bw_k\bv_j$ for $k\geq 0$ and $t_k\leq j < t_{k+1}$.
\label{Reciproque Propriete 2 sur y_k}
\smallskip
\item $\bv_{j} = \bw_k\bv_{\psi(j)}$ for $k\geq 1$ and $t_k\leq j < t_{k+1}$.
\label{Reciproque Propriete 4 sur y_k}
\end{enumerate}
Moreover, the sequence $(\bw_k)_{k\geq 0}$ is a $\psi$-Sturmian sequence in $\GL(\bR)$.
\end{proposition}

\begin{proof}
Since $(t_0,t_1)=(-1,0)$, the case $k=0$ of \ref{Reciproque Propriete 2 sur y_k} is trivial by definition of $\bw_0$. Let $k,j$ be integers with $k\geq 1$ and $t_k\leq j<t_{k+1}$. Recall that $\psi(j)=j-1$ for each $j$ with $t_k < j < t_{k+1}$, so that, using successively \eqref{y_psi == à [y_i,y_i,y_i+1]}, we find
\begin{equation*}
    \bv_{j+1}\bv_{j}^{-1} = \bv_{j}\bv_{\psi(j)}^{-1} = \dots = \bv_{t_k+1}\bv_{t_k}^{-1} = \bw_k = \bv_{t_k}\bv_{\psi(t_k)}^{-1},
\end{equation*}
which proves \ref{Reciproque Propriete 2 sur y_k} and \ref{Reciproque Propriete 4 sur y_k}. Assertion \ref{Reciproque Propriete 1 sur y_k} is a consequence of \ref{Reciproque Propriete 2 sur y_k} and \ref{Reciproque Propriete 4 sur y_k}.\\
Finally, by \eqref{y_psi == à [y_i,y_i,y_i+1]}, we have $\bw_{k+1} = \bv_{t_{k+1}}\bv_{\psi(t_{k+1})}^{-1} = (\bv_{t_{k}+s_{k+1}}\bv_{t_k}^{-1})(\bv_{t_k}\bv_{t_k-1}^{-1})$. Using \ref{Reciproque Propriete 1 sur y_k} and \ref{Reciproque Propriete 2 sur y_k}, we
find $\bw_{k+1}=\bw_k^{s_{k+1}}\bw_{k-1}$, hence the last part of the proposition.
\end{proof}

\begin{proposition}
\label{prop: reconstruction w_k admissible}
Let $(\bv_i)_{i\geq -2}$ be as in Proposition~\ref{prop: reconstruction suite sturm}, and define the $\psi$-Sturmian sequence $(\bw_k)_{k\geq 0}$ as in Proposition~\ref{Reciproque Propriete sur y_k}. Setting $N:=\transpose{(\bv_{-1}\bv_{0}^{-1}\bv_{-2})}$, we have $\bv_{-2}=\bw_0\transpose{N}$ and
\begin{equation}
    \label{eq: reconstruction y_i en fonction w_k et N}
    \bv_{t_k+\ell} = \bw_k^{\ell+1}\bw_{k-1}N_k\quad (k\geq 1 \textrm{ and } 0\leq \ell < s_{k+1}),
\end{equation}
where $N_k=N$ if $k$ is even, $N_k = \transpose N$ else.
\end{proposition}

\begin{proof}
For each $k\geq 0$, set
\begin{equation}
\label{reciproque Def N_k EQ}
    N_{k+1}' := \bw_k^{-1}\bv_{\psi(t_{k+1})} = \bv_{t_k}\bv_{t_k+1}^{-1}\bv_{t_{k}-1}.
\end{equation}
By \eqref{y_psi == à [y_i,y_i,y_i+1]}, we have
\begin{align*}
    \bv_{t_{k+1}}\bv_{t_{k+1}+1}^{-1} = \bv_{\psi(t_{k+1})}\bv_{t_{k+1}}^{-1} \AND \bv_{t_{k+1}}^{-1}\bv_{t_{k+1}-1} = \cdots = \bv_{t_{k}+1}^{-1}\bv_{t_{k}},
\end{align*}
from which we deduce
\begin{align*}
    N_{k+2}' =  \bv_{t_{k+1}}\bv_{t_{k+1}+1}^{-1}\bv_{t_{k+1}-1} =
    \bv_{\psi(t_{k+1})}\bv_{t_{k+1}}^{-1}\bv_{t_{k+1}-1}= \bv_{t_{k}-1}\bv_{t_{k}+1}^{-1}\bv_{t_{k}} = \transpose{N_{k+1}'}.
\end{align*}
Since $\transpose{N} = N_1=N_1'$, it implies that $N_k=N_k'$ for each $k\geq 1$, and \eqref{reciproque Def N_k EQ}
provide the identity $\bv_{\psi(t_{k+1})} = \bw_kN_{k+1}$. With $k=0$, this gives $\bv_{-2} = \bw_0\transpose{N}$. More generally,
combined with Proposition~\ref{Reciproque Propriete sur y_k}, this yields  \eqref{eq: reconstruction y_i en fonction w_k et N}.
\end{proof}

\subsection{New key-identities} ~ \medskip
\label{subsection: fonction U}

Recall that $\bR^3$ is identified to the space of symmetric matrices of $\Mat(\bR)$, so that $X\wedge Y$ is well defined for any symmetric matrices $X,Y\in \Mat(\bR)$. We denote by $J$ the matrix defined as in \eqref{Eq def matrice J}. The goal of this section is to give another expression for the sequence $(\bz_i)_{i\geq -1}$ of Definition~\ref{Def (z_i)_i}. This will allow us to compute the exponents $\omega_2^*$ and $\homega_2^*$ of a Sturmian number in Section~\ref{subsection: exposants omega^*}. See the introduction and \eqref{eq intro: U version polynomial} for the motivation of the following definition.

\begin{definition}
We define the morphism $\UU:\Mat(\bR)\rightarrow\Mat(\bR)$ by
\[
    \UU\left(\left(\begin{array}{cc}  a & b \\ c & d \end{array}\right)\right) :=  \left(
    \begin{array}{cc}
        -c & a-d \\
        a-d & b
    \end{array}
    \right).
\]
\end{definition}

Note that for each $X\in\Mat(\bR)$, the matrix $\UU(X)$ is symmetric, $ U(\Adj(X)) = -U(X)$, and $\UU(X)=0$ if and only if $X$ is proportional
to $\Id$.

\begin{definition}
\label{def: def alternative y_i et z_i}
Let $(\bw_k)_{k\geq 0}$ be a $\psi$-Sturmian sequence in $\GL(\bR)$ and set $\bw_{-1} := \bw_0^{-1}\bw_1$. We associate to $(\bw_k)_{k\geq 0}$ two sequences $(\ba_i)_{i\geq -1}$ and $(\bb_i)_{i\geq -1}$ of symmetric matrices as follows. For $k,\ell\in\bN$ with $k\geq 0$ and $0\leq \ell <s_{k+1}$, we define $\bb_{t_k+\ell} = U(\bw_k^\ell\bw_{k-1})$ and
\begin{equation*}
    \ba_{t_k+\ell}=(-1)^{k+1}\bb_{t_{k+1}}\wedge\bb_{t_k+\ell}=(-1)^{k+1}U(\bw_k)\wedge U(\bw_k^\ell\bw_{k-1}).
\end{equation*}
\end{definition}

The main result of this section is the following.

\begin{proposition}
\label{Prop nouvelle expression y_i et z_i}
Let $(\bw_k)_{k\geq 0}$ be an admissible $\psi$-Sturmian sequence in $\GL(\bR)$. Then
\begin{equation*}
    \ba_i= \by_i \AND \bb_i=\det(N)^{-1} \bz_i
\end{equation*}
for each $ i\geq -1$, where $(\by_i)_{i\geq -2}$ and $(\bz_i)_{i\geq -1}$ are the sequences of symmetric matrices given by Definition~\ref{Def (y_i)} with $N= (\Id-\bw_1^{-1}\bw_0^{-1}\bw_1\bw_0)J$.
\end{proposition}

Before proving this result, let us state some elementary identities satisfied by $U$. They can easily be obtained by a direct computation, details are left to the reader.

\begin{proposition}
    \label{Prop: proprietes de U}
    For each $X,Y\in\Mat(\bR)$, we have the following properties.
    \begin{enumerate}[label=\rm(\roman*)]
        \item $\UU(X)\wedge\UU(Y)=0$ if and only if $XY=YX$. More precisely
        \begin{equation}
        \label{eq lien U et commutateur}
            \UU(X)\wedge\UU(Y) = -(XY-YX)J.
        \end{equation}
        \item \label{eq fnct U pour X^n*Y} If $X,Y$ are inversible, then $\UU(X)\wedge\UU(Y) = -XY(\Id-Y^{-1}X^{-1}YX)J$.
        \smallskip
        \item \label{enumi: Prop propriete 3} If  $X,Y$ are symmetric, then $U(X\Adj(Y)) = -X\wedge Y$. 
    \end{enumerate}
\end{proposition}

It is also interesting to notice the following identities (although we will not need them in this paper)
\[
    U(XY)+U(YX)=\Tr(X)U(Y)+\Tr(Y)U(X),
\]
and
\begin{equation*}
    \UU(PXP^{-1})\wedge\UU(PYP^{-1}) = \det(P)^{-1} P \big(\UU(X)\wedge\UU(Y)\big)\transpose P,
\end{equation*}
valid for each $X,Y\in\Mat(\bR)$ and $P\in\GL(\bR)$. We get the first one by a direct computation, and the last one is a consequence of \eqref{eq lien U et commutateur}
and the equality $\Adj(P)J=J(\transpose{P})$.

\bigskip

\begin{proof}[Proof of Proposition~\ref{Prop nouvelle expression y_i et z_i}]
Let $k,\ell$ with $k\geq 0$ and $0\leq \ell <s_{k+1}$. We first prove the formula $\bb_i=\det(N)^{-1} \bz_i$ for each $i\geq -1$. We can derive from \eqref{eq: condition symetrie admissible} the general identity $\bw_{k}\bw_{k-1}N_k = \bw_{k-1}\bw_{k}N_{k+1}$ (see \cite[Proposition 3.4]{poels2017exponents}). Combined with \eqref{Eq Def y_i}, we obtain  $\by_{t_k+\ell} = \bw_k^{\ell}\bw_{k-1}\bw_{k}N_{k+1}$. On the other hand Eq. \eqref{Eq y_psi(k) = w_k-1} gives $\by_{\psi(t_{k+1})} = \bw_kN_{k+1}$, so that
\begin{align*}
    \by_{t_k+\ell}\by_{\psi(t_{k+1})}^{-1} = \bw_k^\ell\bw_{k-1}.
\end{align*}
Since $\by_{t_k+\ell}$ and $\by_{\psi(t_{k+1})}$ are symmetric, together with assertion~\ref{enumi: Prop propriete 3} of Proposition~\ref{Prop: proprietes de U} and \eqref{Eq Def y_i}, this yields
\begin{align*}
    \bb_{t_k+\ell} &= U(\bw_k^\ell\bw_{k-1}) = \det(\by_{\psi(t_{k+1})})^{-1}\by_{\psi(t_{k+1})}\wedge \by_{t_k+\ell} = \det(N)^{-1}\bz_{t_k+\ell}.
\end{align*}
Now we prove that $\ba_{t_k+\ell} = \by_{t_k+\ell}$. Assertion \ref{eq fnct U pour X^n*Y} of Proposition~\ref{Prop: proprietes de U} gives
\begin{align*}
    \ba_{t_k+\ell} = (-1)^{k+1}U(\bw_k)\wedge U(\bw_k^\ell\bw_{k-1})
    = (-1)^k\bw_k^{\ell+1}\bw_{k-1}(\Id-\bw_{k-1}^{-1}\bw_{k}^{-1}\bw_{k-1}\bw_{k})J.
\end{align*}
We conclude by noticing that $\bw_{k+1}^{-1}\bw_{k}^{-1}\bw_{k+1}\bw_{k} = \bw_{k-1}^{-1}\bw_{k}^{-1}\bw_{k-1}\bw_{k}$ combined with \eqref{eq fnct U pour 2nd expr wedge} implies
\begin{align*}
    (-1)^k(\Id-\bw_{k-1}^{-1}\bw_{k}^{-1}\bw_{k-1}\bw_{k})J = N_k.
\end{align*}
\end{proof}

\section{Estimates for Sturmian sequences of matrices}
\label{section: estimates for Sturmian sequences}

We keep the notation of Section~\ref{section: combinatorics of Sturmian seq} for $\us$ and $\psi=\psi_\us$. In §\ref{subsection: multi growth}, we establish a new simple criterion so that a given admissible $\psi$-Sturmian sequence $(\bw_k)_{k\geq 0}$ has multiplicative growth.
In §\ref{subsection: estimate contents}, we solve the delicate question (and essential for our study) of knowing how to control the content of $\bw_k$, assuming that $(\bw_k)_{k\geq 0}$ is defined over $\bQ$. Altogether with the results of the previous section, we finally establish a new characterization of $\Sturm(\us)$ in §\ref{subsection: new characterization of Sturm}.

\medskip

The next result will allow us to eliminate the degenerate situation where a $\psi$-Sturmian sequence is admissible with an antisymmetric matrix $N$.

\begin{lemma}
\label{lem: matrice N = J}
    Let $(\bw_k)_{k\geq 0}$ be a $\psi$-Sturmian sequence in $\GL(\bR)$ and $N\in\GL(\bR)$ be such that $\bw_0\transpose{N}$, $\bw_1N$ and $\bw_1\bw_0\transpose{N}$ are symmetric. Suppose that the sequence $(\by_i)_{i\geq -2}$ of symmetric matrices associated to $(\bw_k)_{k\geq 0}$ and $N$ as in Definition~\ref{Def (y_i)} converges projectively. Then $N$ is not    antisymmetric.
\end{lemma}

\begin{proof}
By contradiction, suppose that $N$ is antisymmetric, and write $N=\alpha J$ with $\alpha\in\bR\setminus\{0\}$. Then, by \eqref{eq: lien w0w1 w1w0}, we have $\bw_0\bw_1=-\bw_1\bw_0 \neq \bw_1\bw_0$. We claim that $\bw_0^\ell$ (resp. $\bw_1^\ell$) is proportional to $\Id$ if $\ell$ is even, and $\bw_0$ (resp. $\bw_1$) if $\ell$ is odd. Indeed, since $J^{-1} = -J = \transpose{J}$, we have
$\alpha\bw_0=\by_{-2}J$ and $\alpha\bw_1=-\by_{-1}J$, and we conclude with the identity $\bx J\bx J = -\det(\bx) \Id$ valid for each symmetric matrix $\bx\in\GL(\bR)$. As a consequence, for any $i\geq -2$, the non-zero symmetric matrix $\by_i$ is proportional to either $\bw_0J$, $\bw_1J$ or $\bw_1\bw_0J$. Since by Lemma~\ref{lem: degenerate N symmetric} the points $\by_{t_k-1}$, $\by_{t_k}$, $\by_{t_k+1}$ are linearly independent for each $k\geq 0$, we deduce that projectively, the sequence $(\by_i)_{i\geq -2}$ has exactly three accumulation points, a contradiction.
\end{proof}

\subsection{Multiplicative growth property} ~ \medskip
\label{subsection: multi growth}

Showing the multiplicative growth of an admissible $\psi$-Sturmian sequence in $\GL(\bR)$ is difficult, partly because of the lack of control of the signs of the coefficients: opposite terms can cancel out. The proof of Lemma~5.1 of \cite{roy2007two} gives a useful criterion for showing the multiplicative growth if $\bw_0$ and $\bw_1$ are of a certain type. The examples given by Roy in \cite{roy2007two} (see also \cite[Section~8.1]{poels2017exponents}) satisfy this criterion and allow us to avoid the alluded difficulty (see also Example~2 of \cite{Roy2004SmallDegree} for an example of construction of extremal numbers which does not satisfies the criterion of \cite[Lemma~5.1]{roy2007two}). We establish a new condition under which an admissible $\psi$-Sturmian sequence has multiplicative growth. Recall that the matrix $J$ is defined by \eqref{Eq def matrice J}.

\begin{proposition}
\label{reciproque Prop w_k croissance multi}
Let $(\bw_k)_{k\geq 0}$ be an admissible $\psi$-Sturmian sequence in $\GL(\bQ)$ and let $(\by_i)_{i\geq -2}$ be the sequence of symmetric matrices associated to $(\bw_k)_{k\geq 0}$ by Definition~\ref{Def (y_i)}. If $(\by_i)_{i\geq -2}$ converges projectively to a point $\by =(1,\xi,\xi^2)$, where $\xi\in\bR$ is neither rational nor quadratic, then $(\bw_k)_{k\geq 0}$ has multiplicative growth.
\end{proposition}

Proposition~\ref{reciproque Prop w_k croissance multi} is a corollary of Proposition~\ref{Reciproque Prop ||prod y_jB_j || = prod ||y_j||} below.

\begin{lemma}
\label{Reciproque Lemme (prod UB_i) != 0 }
Let $\xi$ be a real number neither rational nor quadratic. Let $N$ be a positive integer and $B_1,\dots,B_N\in\Mat(\bQ)\setminus\big(\bQ J\big)$. Then $\Xi B_N\Xi B_{N-1}\dots \Xi B_1 \neq 0$, where $\Xi$ is the symmetric matrix corresponding to $(1,\xi,\xi^2)$.
\end{lemma}

\begin{proof}
Since the image of $\Xi$ is equal to $\Span{\bx}$, where $\bx=\transpose(1,\xi)$, it suffices to prove that for any $B\in \Mat(\bQ)\setminus\big(\bQ J\big)$, the vector $\Xi B\bx \in \Span{\bx}$ is non-zero. Let $B\in \Mat(\bQ)$ and write
\[
    B = \left(\begin{array}{cc} a & b\\ c & d \end{array}\right).
\]
Then $\Xi B\bx = 0$ if and only if $a+(b+c)\xi+d\xi^2 = 0$. Since $\xi$ is neither rational nor quadratic, it is equivalent to $B\in \bQ J$.
\end{proof}

\begin{proposition}
\label{Reciproque Prop ||prod y_jB_j || = prod ||y_j||}
Let $(\bw_k)_{k\geq 0}$ and $(\by_i)_{i\geq -2}$ satisfying the hypotheses of Proposition~\ref{reciproque Prop w_k croissance multi}. Let $N$ be a positive integer and $\mathcal{B}$ be a finite subset of $\Mat(\bQ)\setminus\big(\bQ J\big)$. Then, there are a constant $c>0$ and an index $i_0$ which only depend on $\mathcal{B}$, $N$ and $\xi$, such that, for any indices $j_1,\dots,j_N\geq i_0$ and any matrices $B_{1},\dots,B_N\in\mathcal{B}$, we have
\begin{equation}
\label{Reciproque eq inter 1 Cor ||prod y_jB_j || = prod ||y_j||}
    c^{-1}\prod_{k=1}^N\norm{\by_{j_k}} \leq \norm{\by_{j_1}B_1\cdots\by_{j_N}B_N} \leq c\prod_{k=1}^N\norm{\by_{j_k}}.
\end{equation}
Furthermore, the sequence $(\bw_k)_{k\geq 0}$ has multiplicative growth.
\end{proposition}

\begin{proof}
Since $N<+\infty$ and $\mathcal{B}$ is finite, the set of $N$-tuples of matrices $(B_1,\dots,B_N)\in\mathcal{B}^N$ is finite. We denote by $\Xi$ the symmetric matrix corresponding to $(1,\xi,\xi^2)$. By Lemma~\ref{Reciproque Lemme (prod UB_i) != 0 }, there is a constant $c_1>0$ (which depends only on $\mathcal{B}$, $N$ and $\xi$), such that
\begin{equation}
\label{Reciproque Lemme croissance multi eq inter 1}
    \frac{1}{c_1} \leq \norm{\Xi B_{1}\cdots \Xi B_{N}} \leq c_1
\end{equation}
for each $(B_1,\dots,B_N)\in\mathcal{B}^N$. Write $\by_i= \left(\begin{array}{cc} y_{i,0} & y_{i,1}\\ y_{i,1} & y_{i,2} \end{array}\right)$ for each $i\geq -2$.
By hypothesis $y_{i,0}^{-1}\by_i$ tends to $\Xi$, in particular $|y_{i_0}| \asymp \norm{\by_i}$. Fix $(B_1,\dots,B_N)\in\mathcal{B}$. From the above the matrix product
\[
    \Big(\prod_{k=1}^Ny_{j_k,0}\Big)^{-1}\by_{j_1}B_{1}\cdots \by_{j_N}B_{N}
\]
tends to $\Xi B_{1}\cdots \Xi B_{N}$ as $i$ tends to infinity, uniformly in $j_1,\dots,j_N\geq i$. Thus, by \eqref{Reciproque Lemme croissance multi eq inter 1}, there exist a constant $c>0$ and an index $i_0$ such that Eq. \eqref{Reciproque eq inter 1 Cor ||prod y_jB_j || = prod ||y_j||} holds for each $j_1,\dots,j_N\geq i_0$. Since $\mathcal{B}^N$ is finite, we may suppose that this estimate is satisfied for all $(B_1,\dots,B_N)\in\mathcal{B}^N$, which ends the proof of \eqref{Reciproque eq inter 1 Cor ||prod y_jB_j || = prod ||y_j||}. We now prove that $(\bw_k)_{k\geq 0}$ has multiplicative growth. Fix $k,\ell$ with $k\geq 1$ and $0\leq \ell \leq s_{k+1}$. By \eqref{Eq y_psi(k) = w_k-1}, we have $\bw_k=\by_{\psi(t_{k+1})}N_{k+1}^{-1}$ and $\bw_k^{\ell}\bw_{k-1} = \by_{t_k+\ell-1}N_k^{-1}$ if $\ell > 0$. In particular  $\norm{\bw_k}\asymp \norm{\by_{\psi(t_{k+1})}}$ and $\norm{\bw_k^\ell\bw_{k-1}}\asymp \norm{\by_{t_k+\ell-1}}$ if $\ell > 0$. Moreover, the matrices $N^{-1}$ and $\transpose{N}^{-1}$ are not proportional to $J$ according to Lemma~\ref{lem: matrice N = J}. By writing
\begin{align*}
       \bw_{k}^{\ell+1}\bw_{k-1} = \bw_k(\bw_k^{\ell}\bw_{k-1}) = \left\{ \begin{array}{ll}
            \by_{\psi(t_{k+1})}N_{k+1}^{-1}\by_{t_k+\ell-1}N_k^{-1} & \textrm{if $\ell > 0$},\\
            \by_{\psi(t_{k+1})}N_{k+1}^{-1}\by_{\psi(t_{k})}N_{k}^{-1} & \textrm{if $\ell = 0$},
        \end{array} \right.
\end{align*}
and by using \eqref{Reciproque eq inter 1 Cor ||prod y_jB_j || = prod ||y_j||}, we conclude easily that $\norm{\bw_k^{\ell+1}\bw_{k-1}}\asymp \norm{\bw_k}\norm{\bw_k^{\ell}\bw_{k-1}}$.
\end{proof}

We can deduce from the proofs of \cite[Proposition~6.1 and Proposition~6.5]{poels2017exponents} the following result (note that in \cite{poels2017exponents}
we suppose that $\bw_k\in\Mat(\bZ)$, but this hypothesis is not needed to get the estimates of our proposition).

\begin{proposition}
    \label{prop: estimations pour suites y_i et z_i}
    Let $(\bw_k)_{k\geq 0}$ be an unbounded admissible $\psi$-Sturmian sequence in $\GL(\bR)$ with multiplicative growth, and denote by $(\by_i)_{i\geq -2}$ and $(\bz_i)_{i\geq -1}$ the associated sequences of symmetric matrices (see Definition~\ref{Def (y_i)}). Suppose that there exists $\delta < 2$ such that $|\det(\bw_k)| \ll \norm{\bw_k}^{\delta}$ for each $k\geq 0$. Then, the sequence $(\by_i)_{i\geq -2}$ converges projectively to a point $\Xi=(1,\xi,\xi^2)$, and
    \begin{align*}
        \norm{\by_i\wedge\Xi} \asymp \frac{|\det(\by_i)|}{\norm{\by_i}}, \quad \norm{\bz_i}\asymp \norm{\by_{\psi(i)}} \AND |\bz_i\cdot \Xi| \asymp \frac{|\det(\by_i)|}{\norm{\by_{i+1}}}.
    \end{align*}
\end{proposition}

\subsection{Estimates for the norms and the contents} ~ \medskip
\label{subsection: estimate contents}

Proposition \ref{Annexe Prop existence delta général} below generalizes, among others, the first part of Proposition~5.6 of \cite{poels2017exponents}. The growth of the contents \eqref{eq: det = norm puiss rho q_k} was originally proven in~\cite[Chapter~3]{poelsPhD} in a different way. This is one of the most delicate points. We are grateful to Damien Roy for pointing us out a much shorter proof than the original one. Recall that the sequence $\us$ is not necessarily bounded and that $\UU$ denotes the map introduced in Section~\ref{subsection: fonction U}. We define the sequence $(p_k)_{k\geq -1}$ by
\begin{equation}
    \label{eq def: suite q_k}
    (p_{-1},p_0) = (0,1) \AND p_{k+1} = s_{k+1}p_k+p_{k-1} \quad (k\geq 0).
\end{equation}

\begin{proposition}
    \label{Annexe Prop existence delta général}
    Let $(\bw_k)_{k\geq 0}$ be a $\psi$-Sturmian sequence in $\GL(\bR)$ with multiplicative growth such that $(\norm{\bw_k})_{k\geq 0}$ is unbounded. Then, there exist real numbers $\alpha,\beta$, with $\beta > 0$, such that
    \begin{equation}
        \label{eq: det, norme = puiss q_k}
        |\det(\bw_k)| \asymp e^{\alpha p_k}\AND \norm{\bw_k} \asymp e^{\beta p_k}=: W_k,
    \end{equation}
    as $k\geq 0$ tends to infinity. If $\bw_k$ is defined over $\bQ$ for each $k\geq 0$, then there is $\rho\in\bR$ such that
    \begin{equation}
        \label{eq: det = norm puiss rho q_k}
        \cont(\bw_k) = e^{p_k(\rho + o(1))},
    \end{equation}
    as $k$ tends to infinity. Suppose furthermore $(\bw_k)_{k\geq 0}$ admissible, and that either $\us$ is bounded or $\alpha = 2\rho$. Given $i=t_k+\ell$ with $k\geq 1$ and $0\leq \ell < s_{k+1}$, we define
    \begin{equation*}
        Y_i:= W_k^{\ell+1}W_{k-1} \AND Z_i = W_k^\ell W_{k-1}.
    \end{equation*}
    Then, as $i$ tends to infinity, we have $\norm{\bw_k^{\ell+1}\bw_{k-1}} = Y_i^{1+o(1)}$, $\norm{\bw_k^{\ell}\bw_{k-1}} = Z_i^{1+o(1)}$, as well as
    \begin{equation}
    \label{eq: prop estimations continues avec ell et UU}
        \cont(\bw_k^{\ell+1}\bw_{k-1}) = Y_i^{\rho/\beta + o(1)} \AND \cont(\UU(\bw_k^\ell\bw_{k-1})) = Z_i^{\rho/\beta + o(1)}.
    \end{equation}
\end{proposition}

The proof of this result is at the end of this section. With that goal in mind, let us introduce for each $i\in\bN$ the sequence $(q^{(i)}_{k})_{k\in\bZ}$, defined by
\begin{equation*}
    q^{(i)}_{k} =
    \left\{
        \begin{array}{ll}
             0 & \textrm{if $k<i$} \\
             1 & \textrm{if $k=i$} \\
             s_kq^{(i)}_{k-1}+ q^{(i)}_{k-2} & \textrm{if $k> i$}.
        \end{array}
    \right.
\end{equation*}
Note that $(q^{(0)}_{k})_{k\geq -1}$ is the sequence $(p_k)_{k\geq -1}$ of
\eqref{eq def: suite q_k}. Moreover,
the theory of continued fractions (see for instance \cite[Chapter I]{schmidt1996diophantine})
ensures that for each $k\geq i > 0$ we have
\begin{equation}
    \label{eq: continued fraction xi_i}
    \frac{q^{(i-1)}_{k}}{q^{(i)}_{k}} = [s_{i};s_{i+1},\dots, s_k] \AND |q^{(i-1)}_{k} - \xi_i q^{(i)}_{k} |
    \leq \frac{1}{q^{(i)}_{k+1}},
\end{equation}
where $\xi_i:=[s_{i};s_{i+1},\dots]$. Note that the right-hand side of \eqref{eq: continued fraction xi_i} still holds (and is an equality) for $k=i-1$. Furthermore, $\xi_i\geq s_i \geq 1$ and $\xi_i=s_i+1/\xi_{i+1}$, hence
\begin{align}
    \label{eq: minoration x_ix_i+1}
    \xi_i\xi_{i+1} = s_i\xi_{i+1} + 1 \geq 2 \quad (i\geq 1).
\end{align}

\begin{lemma}
\label{lem:lemme sur les q_i}
For each $i\geq 0$ the quotient $q^{(i)}_{k}/q^{(0)}_{k}$ tends to $a_i:=1/(\xi_1\dots\xi_i)$ as $k$ tends to infinity, where $a_0=1$. Moreover, $\sum_{i\geq 0} s_ia_i < \infty$ and there exist $A,B > 0$ with the following properties.
\begin{enumerate}[label=\rm(\roman*)]
    \item \label{hyp 1 sur les q_i} $\displaystyle \sum_{i\geq 0} s_i|q^{(i)}_{k}-a_iq^{(0)}_{k}| \leq A$ for each $k\geq 0$;
    \smallskip
    \item \label{hyp ii sur les q_i} $q^{(i)}_{k}/q^{(0)}_{k} \leq Ba_i$ for each $k,i\geq 0$.
\end{enumerate}
\end{lemma}

\begin{proof}
Let $i\geq 0$. Then $q^{(i)}_{k}/q^{(0)}_{k}=a_i+o(1)$ as $k$ tends to infinity, since for each $j\geq 1$ the quotient $q^{(j-1)}_{k}/q^{(j)}_{k}$ tends to $\xi_j$. For each $a,b\in\bR$, we set $\delta_{a\leq b}=1$ if $a\leq b$, and $0$ else. Given $k \geq 0$, we have
\begin{align*}
    \big|a_iq_k^{(0)}-q_k^{(i)}\big|
    = \Big| \sum_{j=1}^i \frac{1}{\xi_j\cdots \xi_i}\big(q_k^{(j-1)}-\xi_jq_k^{(j)}\big) \Big|
    \leq \sum_{j \geq 1} \frac{\delta_{j\leq i}}{\xi_j\cdots \xi_i}\frac{\delta_{j\leq k+1}}{q_{k+1}^{(j)}},
\end{align*}
where the last estimate is obtained by noticing that the term in the first sum vanishes if $k<j-1$, and by using the upper bound given by \eqref{eq: continued fraction xi_i} for the indices $j$ with $j\leq k+1$. This yields
\begin{align}
\label{eq inter: preuve prop des q_i}
    \sum_{i \geq 0} s_i|a_iq^{(0)}_{k}-q^{(i)}_{k}| \leq
    \sum_{i \geq 0} \sum_{j \geq 1} s_i\frac{\delta_{j\leq i}}{\xi_j\cdots \xi_i}\frac{\delta_{j\leq k+1}}{q_{k+1}^{(j)}} =
    \sum_{j = 1}^{k+1} \frac{1}{q_{k+1}^{(j)}}\sum_{i \geq j} \frac{s_i}{\xi_j\cdots \xi_i}.
\end{align}
On the one hand, using the upper bound $s_i\leq \xi_i$ and \eqref{eq: minoration x_ix_i+1}, we get
\begin{align*}
    \sum_{i \geq j} \frac{s_i}{\xi_j\cdots \xi_i} \leq \sum_{i\geq j} \frac{1}{\xi_j\cdots \xi_{i-1}}
    \leq \sum_{i\geq j} \frac{1}{2^{\lfloor(i-j)/2\rfloor}} =  2 \sum_{k\geq 0} \frac{1}{2^{k}} = 1.
\end{align*}
Taking $j=1$ in the above, we get $\sum_{i \geq 1} s_ia_i \leq 1$. On the other hand, since the golden ratio $\gamma = (1+\sqrt 5)/2$ satisfies $\gamma^{j} = \gamma^{j-1}+\gamma^{j-2}$ for each $j\geq 0$, it is easy to check by induction that $q^{(j)}_{k+1} \geq \gamma^{k-j}$ for each $j\leq k+1$. Together with \eqref{eq inter: preuve prop des q_i}, we find
\begin{align*}
    \sum_{i \geq 0} s_i|a_iq^{(0)}_{k}-q^{(i)}_{k}| \leq \sum_{j=1}^{k+1} \frac{1}{\gamma^{k-j}}
    \leq \sum_{j\geq 0} \frac{1}{\gamma^{j-1}} =:A < +\infty,
\end{align*}
hence \ref{hyp 1 sur les q_i}. We now prove \ref{hyp ii sur les q_i}. We may assume that $k\geq i$ since $q_k^{(i)}=0$ if $k<i$. If $k > i$, then we have $[s_i;s_{i+1},\dots,s_k] \geq  \xi_i(1-\gamma^{i-k})$, since and $\xi_i\geq 1$ and $| [s_i;s_{i+1},\dots,s_k] - \xi_i| \leq 1/q^{(i)}_{k+1} \leq \gamma^{i-k}$ by \eqref{eq: continued fraction xi_i}. If $k=i$, we simply have $[s_i;\dots,s_k] = s_i \geq \xi_i/2$. Combining these lower bounds, we find
\begin{align*}
    \frac{q^{(i)}_{k}}{q^{(0)}_{k}} = \frac{1}{[s_1;s_{2},\dots,s_k]\cdots[s_i;s_{i+1},\dots,s_k]} \leq \frac{B}{\xi_1\cdots\xi_i},\quad \textrm{where } B:= \frac{2}{\prod_{j\geq 1}(1-\gamma^{-j})} < \infty.
\end{align*}
\end{proof}

\begin{lemma}
    \label{Lem: formule pratique avec les ee_k contenu etc}
     Let $(r_k)_{k\geq 0}$ be a sequence of real numbers. Set $\ee_k:= r_k - s_kr_{k-1}-r_{k-2}$ for each $k\geq 0$, with $(r_{-2},r_{-1})=(0,0)$.
    Then
    \begin{equation}
    \label{eq: key identity pour preuve rec asymp}
        r_k = \sum_{i=0}^k \ee_i q^{(i)}_k \qquad (k \geq -2), 
    \end{equation}
    with the convention that the sum on the right-hand side is equal to $0$ if $k < 0$.
\end{lemma}

\begin{proof}
    We prove Eq. \eqref{eq: key identity pour preuve rec asymp} by induction on $k\geq -2$.
    It trivially holds for $k=-2$ and $k=-1$. Suppose now that \eqref{eq: key identity pour preuve rec asymp}
    is satisfied for $-2,-1,\dots,k-1$, where $k$ is a given integer $\geq 0$. Then, since $q^{(k)}_k=1$
    and $q^{(k-1)}_{k-2}=0$, we obtain
    \begin{align*}
        r_k - \ee_kq^{(k)}_k = s_kr_{k-1}+r_{k-2} &= \sum_{i=0}^{k-1} \ee_i(s_kq^{(i)}_{k-1} + q^{(i)}_{k-2}) = \sum_{i=0}^{k-1} \ee_iq^{(i)}_k.
    \end{align*}
\end{proof}

\begin{lemma}
    \label{Lem: croissance multi implique croissance en cst*q_0}
    Let $(r_k)_{k\geq 0}$ and $(\ee_k)_{k\geq 0}$ be as in Lemma~\ref{Lem: formule pratique avec les ee_k contenu etc}, and
    suppose that $\ee_k = \GrO(s_k)$. Then, there exists $\lambda\in\bR$ such that
    \begin{equation}
        \label{eq: lemme croissance multi}
        r_k = \lambda p_k+\GrO(1) \qquad (k\geq 0).
    \end{equation}
\end{lemma}

\begin{remark}
    Eq. \eqref{eq: lemme croissance multi} can be deduced from the proof of \cite[Proposition~5.5]{poels2017exponents}.
\end{remark}

\begin{proof}
    Set $\lambda := \sum_{i\geq 0} \ee_ia_i$ and let $c>0$ be such that $|\ee_i|\leq cs_i$ for each $i\geq 0$. We have $\lambda \in \bR$ by Lemma~\ref{lem:lemme sur les q_i}. Then, using \eqref{eq: key identity pour preuve rec asymp} and \ref{hyp 1 sur les q_i} of Lemma~\ref{lem:lemme sur les q_i}, we obtain for each $k\geq 0$
    \begin{align*}
        |r_k-\lambda q^{(0)}_k| = \big|\sum_{i \geq 0} \ee_i(q^{(i)}_k - a_iq^{(0)}_k)\big| \leq c\sum_{i \geq 0} s_i|q^{(i)}_k - a_iq^{(0)}_k| \leq cA.
    \end{align*}
\end{proof}

Surprisingly, under the weaker assumption $\ee_k\geq 0$, the quotient $r_k/p_k$ still converges, as soon as it is bounded (see below).

\begin{lemma}
    \label{Lem: croissance des contenus en cst*q_0}
    Let $(r_k)_{k\geq 0}$ and $(\ee_k)_{k\geq 0}$ be as in Lemma~\ref{Lem: formule pratique avec les ee_k contenu etc}, and
    suppose that for any sufficiently large $k$, we have $\ee_k\geq 0$ and $r_k\leq M p_k$ for a constant $M > 0$ independent of $k$.
    Then $r_k/p_k$ has a limit as $k$ tends to infinity.
\end{lemma}

\begin{proof}
    Let $k,\ell$  be integers with $k\geq \ell\geq 0$. By \eqref{eq: key identity pour preuve rec asymp}, if $\ell$ is large enough, then we have $\ee_i \geq 0$ for each $i\geq \ell$, and
    \begin{align*}
        \sum_{i=0}^\ell \ee_i\frac{q^{(i)}_k}{q^{(0)}_k} \leq \sum_{i=0}^k \ee_i\frac{q^{(i)}_k}{q^{(0)}_k} = \frac{r_k}{q^{(0)}_k} \leq M.
    \end{align*}
    Yet, the quotient $q^{(i)}_k/q^{(0)}_k$ tends to $a_i$. By letting first $k$, then $\ell$, tend to infinity in the left-hand side, this shows that the series $\sum_{i\geq 0}\ee_ia_i$ converges absolutely. So, we can apply the dominated convergence Theorem by \ref{hyp ii sur les q_i} of Lemma~\ref{lem:lemme sur les q_i}. Defining $\delta_{i\leq k} = 1$ if $i\leq k$, and $0$ otherwise, we get
    \begin{align*}
        \lim_{k\rightarrow\infty} \frac{r_k}{q^{(0)}_k} = \lim_{k\rightarrow\infty} \sum_{i \geq 0} \delta_{i\leq k}\ee_i\frac{q^{(i)}_k}{q^{(0)}_k}
        = \sum_{i\geq 0} \lim_{k\rightarrow\infty} \delta_{i\leq k}\ee_i\frac{q^{(i)}_k}{q^{(0)}_k} = \sum_{i\geq 0} \ee_ia_i \in\bR.
    \end{align*}
\end{proof}

\begin{proof}[Proof of Proposition~\ref{Annexe Prop existence delta général}]
Let $(\bw_k)_{k\geq 0}$ be a $\psi$-Sturmian sequence as in Proposition~\ref{Annexe Prop existence delta général}. For each $k\geq 2$, we have $\bw_{k} = \bw_{k-1}^{s_{k}}\bw_{k-2}$. Consequently, the sequence $(r_k)_{k\geq 0} := (\log |\det(\bw_{k})|)_{k\geq 0}$ satisfies $r_k-s_kr_{k-1}-r_{k-2} = 0$ for each $k\geq 2$,
and we obtain the estimate for $|\det(\bw_k)|$ in \eqref{eq: det, norme = puiss q_k} by applying Lemma~\ref{Lem: croissance multi implique croissance en cst*q_0}. Similarly, the multiplicative growth gives
\[
    \log \norm{\bw_{k}}-s_{k}\log \norm{\bw_{k-1}}-\log \norm{\bw_{k-1}} = \GrO(s_{k}),
\]
and therefore, we can also apply Lemma~\ref{Lem: croissance multi implique croissance en cst*q_0} with the sequence $(r_k)_{k\geq 0} := (\log \norm{\bw_k})_{k\geq 0}$. Hence the estimate for $\norm{\bw_k}$ in \eqref{eq: det, norme = puiss q_k}. Note that since $(\norm{\bw_k})_{k\geq 0}$ is unbounded, we must have $\beta > 0$.

\medskip

Now, suppose that $(\bw_k)_{k\geq 0}$ is defined over $\bQ$ and consider $(r_k)_{k\geq 0} := (\log \cont(\bw_k))_{k\geq 0}$. Given $k\geq 2$, the identity $\bw_{k} = \bw_{k-1}^{s_{k}}\bw_{k-2}$ implies that $r_k \geq s_{k}r_{k-1}+r_{k-2}$ by definition of the content. Moreover, Eq. \eqref{eq: det, norme = puiss q_k} yields $r_k \leq \log \norm{\bw_k} \ll \beta p_k$ as $k$ tends to infinity. Lemma~\ref{Lem: croissance des contenus en cst*q_0} gives \eqref{eq: det = norm puiss rho q_k}.

\medskip

Suppose in addition that $(\bw_k)_{k\geq 0}$ is admissible, and that $\us$ is bounded or $\alpha=2\rho$. Let $i=t_k+\ell$ with $k\geq 1$ and $0\leq \ell < s_{k+1}$. The estimates $\norm{\bw_k^{\ell+1}\bw_{k-1}} = Y_i^{1+o(1)}$, $\norm{\bw_k^{\ell}\bw_{k-1}} = Z_i^{1+o(1)}$ are obtained by multiplicative growth (and since $\norm{\bw_k}\asymp W_k$ tends to infinity). Furthermore
\begin{align*}
    \cont(\UU(\bw_k))\cont(\UU(\bw_k^\ell\bw_{k-1})) \leq \cont(\UU(\bw_k)\wedge\UU(\bw_k^\ell\bw_{k-1})) \asymp \cont(\bw_k^{\ell+1}\bw_{k-1}),
\end{align*}
where the last part comes from Proposition~\ref{Prop nouvelle expression y_i et z_i}. On the other hand, $\cont(\bw_k)\leq \cont(\UU(\bw_k))$ and $\cont(\bw_k^\ell\bw_{k-1})\leq \cont(\UU(\bw_k^\ell\bw_{k-1}))$. So the left-hand side of \eqref{eq: prop estimations continues avec ell et UU} implies its right-hand side, and it just remains to prove the estimates with $Y_i$. For each $j\geq 0$, we set $c_j:=\cont(\bw_j)$. By \eqref{eq: det = norm puiss rho q_k} we have $c_k = W_k^{\rho/\beta+o(1)}$. Since $\cont(\bw\bw') \geq \cont(\bw)\cont(\bw')$ for each $\bw,\bw'\in\Mat(\bZ)\setminus\{0\}$, we find
\begin{align}
    \label{eq proof: cont(w_k^ell w_k-1) inter 0}
    \cont(\bw_k^{\ell+1}\bw_{k-1}) \geq c_k^{\ell+1}c_{k-1} = (W_k^{\ell+1}W_{k-1})^{\rho/\beta+o(1)}.
\end{align}

\noindent\textbf{First case.} Suppose that $\us$ is bounded. Then $o(s_{k+1}) = o(1)$ and $W_{k+1}^{o(1)} = W_k^{o(1)}$. We deduce easily \eqref{eq: prop estimations continues avec ell et UU} from \eqref{eq proof: cont(w_k^ell w_k-1) inter 0} and \eqref{eq: det = norm puiss rho q_k} combined with the upper bound
\begin{align*}
    c_{k+1} \geq c_k^{s_{k+1}-\ell-1} \cont(\bw_k^{\ell+1}\bw_{k-1}).
\end{align*}

\noindent\textbf{Second case.} Suppose that $\us$ is unbounded and $\alpha = 2\rho$. The matrix $\bw:=c^{-1}\bw_k^{\ell+1}\bw_{k-1}$ has integer coefficients, where $c:=\cont(\bw_k^{\ell+1}\bw_{k-1})$. Using \eqref{eq: det, norme = puiss q_k} we find
\begin{align*}
    1 \leq |\det(\bw)| \leq c^{-2}(W_k^{\ell+1}W_{k-1})^{\alpha/\beta+o(1)} =  c^{-2} (W_k^{\ell+1}W_{k-1})^{2\rho/\beta+o(1)},
\end{align*}
hence the upper bound $\cont(\bw_k^{\ell+1}\bw_{k-1})\leq Y_i^{\rho/\beta+o(1)}$. We conclude by \eqref{eq proof: cont(w_k^ell w_k-1) inter 0}.
\end{proof}

\subsection{Property of the set $\Sturm$} ~ \medskip
\label{subsection: new characterization of Sturm}

Let $\us$ be a sequence of positive integers and set $\sigma=\sigma(\us)$. It is difficult to compute the Diophantine exponents of an element of $\Sturm(\us)$ using only Definition~\ref{def: Sturm(psi)}. The main result of this section, namely Theorem~\ref{reciproque Prop construction 3-systeme : construction bwu et byu} below, will help dealing with this problem. Recall that $(p_k)_{k\geq -1}$ is defined by \eqref{eq def: suite q_k} and $\UU$ is the morphism introduced in Section~\ref{subsection: fonction U}.

\begin{theorem}
\label{reciproque Prop construction 3-systeme : construction bwu et byu}
    Let $\xi\in\Sturm(\us)$ and set $\Xi=(1,\xi,\xi^2)$. Then, there exist $\beta,\delta\in\bR$ with $\beta > 0$ and $\delta\in[0,\sigma/(1+\sigma)]$ with the following properties. We define $(W_k)_{k\geq 0}:= (e^{\beta p_k})_{k\geq 0}$ and for each $i=t_k+\ell$ with $k\geq 1$ and $0\leq \ell < s_{k+1}$, we set
    \begin{align*}
         Y_i:= W_k^{\ell+1}W_{k-1} \AND Z_i = W_k^\ell W_{k-1}.
    \end{align*}
    Then, there is an admissible $\psi$-Sturmian sequence $(\bw_k)_{k\geq 0}$ in $\GL(\bR)$, with multiplicative growth and defined over $\bQ$, such that
    \begin{enumerate}[label=\rm(\roman*)]
    \item \label{item: prop pour construction 3 system 2}  $\norm{\bw_k} \asymp W_k$ and $|\det(\bw_k)| \asymp W_k^{\delta}$;
    \smallskip
    \item \label{item: prop pour construction 3 system 3}  $\cont(\bw_k)= W_k^{o(1)}$ as $k$ tends to infinity.
    \end{enumerate}
    Moreover, denoting for each $i\geq -1$ by $\byt_i,\bzt_i\in\bZ^3$ the non-zero primitive integer points
    \begin{equation}
        \label{eq:def a_i et b_i}
        \byt_i:=\cont(\ba_i)^{-1}\ba_i\AND \bzt_i:= \cont(\bb_i)^{-1}\bb_i,
    \end{equation}
    where $(\ba_i)_{i\geq -1}$ and $(\bb_i)_{i\geq -1}$ are defined as in Definition~\ref{def: def alternative y_i et z_i}, we have the estimates
    \begin{align}
    \label{eq:thm estimations points primitiv 1}
        \norm{\byt_i} = Y_i^{1+o(1)},\quad \norm{\byt_i\wedge\Xi} = Y_i^{-1+\delta+o(1)}, \quad \norm{\bzt_i} = Z_i^{1+o(1)} \AND |\bzt_i\cdot \Xi| = Y_i^{\delta+o(1)}Y_{i+1}^{-1}
    \end{align}
    as $i$ tends to infinity. In particular, the sequence $(\byt_i)_{i\geq-2}$ converges projectively to $\Xi$.
\end{theorem}

\begin{remark}
    As we will see, the parameter $\delta$ in our theorem depends only on $\xi$ (see the remark after Theorem~\ref{reciproque Thm exposants nb quasi sturmien}).
\end{remark}

\begin{proof}[Proof of Theorem~\ref{reciproque Prop construction 3-systeme : construction bwu et byu}]
By Definition~\ref{def: Sturm(psi)} (and the remarks below Definition~\ref{def: Sturm(psi)}), there is a sequence $(\bv_i)_{i\geq i_0}$ (with $i_0\in\bN)$ of non-zero primitive points in $\bZ^3$ such that $\det(\bv_i)\neq 0$ for each $i\geq i_0$, the point $\bv_{i+1}$ is proportional to $\bv_i\bv_{\psi(i)}^{-1}\bv_i$ for each $i$ with $\psi(i)\geq i_0$, and
\begin{equation}
\label{eq proof: prop new characterization of Sturm: eq0}
    |\det(\bv_{i})| = \norm{\bv_{i}}^{\sigma/(1+\sigma)+o(1)}
\end{equation}
as $i$ tends to infinity with $\psi(i+1) < i$. The space generated by $(\bv_i)_{i\geq i_0}$ has dimension $3$. Indeed, if $(\bv_i)_{i\geq i_0}$ was included in a subspace $V$ of dimension $2$ and defined over $\bQ$, it would imply that $\Xi\in V$, which is impossible since the coordinates of $\Xi$ are linearly independent over $\bQ$. Upon defining $\bv_{i_0-1},\dots,\bv_{-2}$ by using the induction formula $\bv_{\psi(i)} = \bv_i\bv_{i+1}^{-1}\bv_i$, we may assume without loss of generality that $i_0=-2$. Then, by Proposition~\ref{prop: reconstruction suite sturm}, there is an admissible $\psi$-Sturmian sequence $(\bw_k')_{k\geq 0}$ defined over $\bQ$ and such that $\by_i'$ is proportional to $\bv_i$ for each $i\geq -2$, where $(\by_i')_{i\geq -2}$ denotes the sequence associated to $(\bw_k')_{k\geq 0}$ by Definition~\ref{Def (y_i)}.
According to Proposition~\ref{reciproque Prop w_k croissance multi} the sequence $(\bw_k')_{k\geq 0}$ has multiplicative growth. Any $\psi$-Sturmian sequence of first terms $\lambda\bw_0'$ and $\mu\bw_1'$ (with $\lambda,\mu\neq 0$) has the above properties. So, upon replacing $\bw_0$, $\bw_1$ by larger multiples, we may assume that $|\det(\bw_0)|,|\det(\bw_1)| > 1$, which implies that $(|\det(\bw_k')|)_{k\geq 0}$ tends to infinity, and therefore $(\bw_k')_{k\geq 0}$ is unbounded. Then, by Proposition~\ref{Annexe Prop existence delta général}, there are $\alpha,\beta',\rho\in\bR$ with $\beta'>0$, such that
\[
    |\det(\bw_k')| \asymp e^{\alpha p_k},\quad \norm{\bw_k'} \asymp e^{\beta' p_k},\quad \cont(\bw_k') = e^{p_k(\rho + o(1))}
\]
as $k$ tends to infinity. since $(\by_i')_{i\geq -2}$ converges projectively to $\Xi$, if follows from the classical estimates of the determinant (see Section~3 of \cite{davenport1969approximation}) that
\[
    |\det(\by_i')| \ll \norm{\by_i'}\norm{\by_i'\wedge\Xi} = o(\norm{\by_i'}^2).
\]
By taking $i=\psi(t_{k+1})$ and by using \eqref{Eq y_psi(k) = w_k-1}, we obtain $\det(\bw_k')= o(\norm{\bw_k'}^2)$, and thus $\alpha < 2\beta'$. Moreover, by definition of the content, the matrix $\cont(\bw_k')^{-1}\bw_k'$ has integer coefficients, so its determinant is a non-zero integer. Consequently, we have $2\rho \leq \alpha$. This leads us to $\beta:=\beta'-\rho > 0$. For each $k\geq 0$, we set
\begin{equation*}
    \bw_k := e^{-\rho p_k}\bw_k'.
\end{equation*}
Then  $(\bw_k)_{k\geq 0}$ is a $\psi$-Sturmian sequence in $\GL(\bR)$ defined over $\bQ$, admissible, with multiplicative growth, and we have
\begin{equation*}
     |\det(\bw_k)| \asymp e^{(\alpha-2\rho) p_k},\quad \norm{\bw_k} \asymp e^{\beta p_k},\quad \cont(\bw_k) = e^{o(p_k)} = \norm{\bw_k}^{o(1)},
\end{equation*}
in particular $(\norm{\bw_k})_{k\geq 0}$ tends to infinity since $\beta > 0$, and $|\det(\bw_k)| \asymp \norm{\bw_k}^\delta$ with $\delta:= (\alpha-2\rho)/\beta \in [0,2)$. It proves \ref{item: prop pour construction 3 system 2} and \ref{item: prop pour construction 3 system 3}. Note that by proposition~\ref{Prop nouvelle expression y_i et z_i}, up to multiplication by a constant, the sequences $(\ba_i)_{i\geq -1}$ and $(\bb_i)_{i\geq -1}$ are the sequences of Definition~\ref{Def (y_i)}. Consequently
\begin{equation}
\label{eq proof: prop new characterization of Sturm: eq3}
    |\det(\ba_i)| \asymp |\det(\bw_k)|^{\ell+1}|\det(\bw_{k-1})|
\end{equation}
for $i=t_k+\ell$ with $k\geq 1$ and $0\leq \ell < s_{k+1}$, and the point $\ba_i$ is proportional to $\by_i'$ (and thus to $\bv_i$). We deduce that $(\ba_i)_{i\geq 0}$ converges projectively to $\Xi$, and Proposition~\ref{prop: estimations pour suites y_i et z_i} yields
\begin{align*}
    \norm{\ba_i\wedge\Xi} \asymp \frac{|\det(\ba_i)|}{\norm{\ba_i}} \AND |\bb_i\cdot \Xi| \asymp \frac{|\det(\ba_i)|}{\norm{\ba_{i+1}}}.
\end{align*}
By \ref{item: prop pour construction 3 system 2} and \eqref{eq proof: prop new characterization of Sturm: eq3}, we have $|\det(\ba_i)| = Y_i^{\delta+o(1)}$. Also note that $Y_{i+1}^{o(1)} = Y_{i}^{o(1)}$ since $Y_{i+1}=W_kY_i\leq Y_i^2$. Putting the above estimates together with
\begin{align*}
    \norm{\ba_i} = Y_i^{1+o(1)}, \quad \norm{\bb_i} = Z_i^{1+o(1)}, \quad \cont(\ba_i)= Y_i^{o(1)} \AND \cont(\bb_i) = Z_i^{o(1)}
\end{align*}
coming from Proposition~\ref{Annexe Prop existence delta général}, we get \eqref{eq:thm estimations points primitiv 1}. Note that $\by_i=\pm\bv_i$ since they are linearly dependent primitive integer points. We obtain $\delta \leq \sigma/(1+\sigma)$ by combining \eqref{eq proof: prop new characterization of Sturm: eq0} with
\begin{align*}
    |\det(\bv_i)| = |\det(\by_i)| \ll \norm{\by_i\wedge\Xi}\norm{\by_i} = Y_i^{\delta+o(1)} = \norm{\bv_i}^{\delta+o(1)}.
\end{align*}
\end{proof}

Now, let us briefly recall the definition of Sturmian type numbers constructed in \cite{poels2017exponents}.

\begin{definition}
\label{def:sturmian type numbers}
Suppose $\us$ bounded. A \textsl{proper $\psi$-Sturmian number} is a real number $\xi$ such that there are a real number $\delta\in[0,\sigma/(1+\sigma))$ and an admissible $\psi$-Sturmian sequence $(\bw_k)_{k\geq 0}$ of matrices in $\GL(\bQ)\cap\Mat(\bZ)$ with the following properties. The sequence of symmetric matrices $(\by_i)_{i\geq -2}$ associated to $(\bw_k)_{k\geq 0}$ by Definition~\ref{Def (y_i)} converges projectively to $(1,\xi,\xi^2)$ and $(\cont(\by_i))_{i\geq -2}$ is bounded. Moreover $(\bw_k)_{k\geq 0}$ is unbounded, has multiplicative growth and satisfies $|\det(\bw_k)|\asymp \norm{\bw_k}^\delta$. The set of \textsl{Sturmian type numbers} is the union of the sets of proper $\psi_\us$-Sturmian numbers for bounded sequences $\us$.
\end{definition}

\begin{remark}
The elements of $\Sturm(\us)$ (when $\us$ is bounded and $\delta < \sigma/(1+\sigma)$) have a lot in common with proper $\psi$-Sturmian numbers. However, a major difference is the possible existence of non-trivial contents for the sequences involved in Theorem~\ref{reciproque Prop construction 3-systeme : construction bwu et byu}. Also note that in our theorem, $(\bw_k)_{k\geq 0}$ does not necessarily have integer coefficients.
\end{remark}

\section{Applications to Diophantine approximation}
\label{section: appli to Diophantine approx}

Our proof of Theorem~\ref{reciproque Thm 3-systeme nombre quasi sturmien intro} (see §\ref{subsection: proofs des thm}) relies on parametric geometry of numbers. We recall the elements of the theory that we need in §\ref{subsection geom param des nb}, and in §\ref{subsection: construction 3-syst} we compute the parametric versions of the exponents $\homega$, $\omega$, $\hlambda$, $\lambda,$ associated to a point $\xi\in\Sturm(\us)$ and $\lambdaL(\xi)$. The two remaining exponents $\homega_2^*$ and $\omega_2^*$ are studied separately in §\ref{subsection: exposants omega^*}.

\subsection{Parametric geometry of numbers} ~ \medskip
\label{subsection geom param des nb}

Let $\Xi=(1,\xi,\xi^2)$ where $\xi\in\bR$ is neither rational nor quadratic. In this section we quickly present Schmidt and Summerer's tools of parametric geometry of numbers in dimension $3$ (see \cite{Schmidt2009} and \cite{Schmidt2013}). In the following, the letter $q$ always denotes a positive real number. Our setting is the same as that of \cite{poels2017exponents}, \ie we consider the two following families of symmetric convex bodies:
\begin{equation*}
    \CCC_{\xi}(e^q) := \{\bx\in\bR^3\;;\; \norm{\bx} \leq 1 \textrm{ and } |\bx\cdot\Xi|\leq e^{-q} \}
\end{equation*}
and
\begin{equation*}
    \CCC_{\xi}^*(e^q) := \{\bx\in\bR^3\;;\; \norm{\bx} \leq e^{q} \textrm{ and } \norm{\bx\wedge\Xi}\leq 1 \}.
\end{equation*}
For $j=1,2,3$, the quantity $\lambda_j(q)$ (resp. $\lambda_j^*(q)$) denotes the $j$-th successive minimum of the convex body $\CCC_{\Xi}(e^q)$ (resp. $\CCC_{\Xi}^*(e^q)$ ) with respect to the lattice $\bZ^3$. We also define
\begin{align*}
    \CL_{j}(q) = \log \lambda_j(q),\quad \psi_{j}(q) = \frac{\CL_{j}(q)}{q},\quad
    \po_j = \limsup_{q\rightarrow\infty}\psi_{j}(q),\quad \pu_j = \liminf_{q\rightarrow\infty}\psi_{j}(q),
\end{align*}
as well as the analogous quantities $\CL_{j}^*(q)$, $\psi_{j}^*(q)$, $\po_j^*$, $\pu_j^*$ associated to $\lambda_j^*(q)$. We group these successive minima $\CL_j$ (resp. $\CL_j^*$) into a single map $\bL_{\xi}=(\CL_1,\CL_2,\CL_3)$ (resp. $\bL_{\xi}^*=(\CL_1^*,\CL_2^*,\CL_3^*)$). In the following proposition (cf \cite{Schmidt2009} and \cite{Roy_juin}) we give a classical relation between standard and parametric Diophantine exponents.

\begin{proposition}
    \label{Prop dico exposants}
    Let $\xi$ be a real number which is neither rational nor quadratic. Then
    \begin{align}
        \label{Eq dico exposants}
        \big(\pu_1,\po_1, \pu_3,\po_3\big) =
        \Big(\frac{1}{\omega_2(\xi)+1},\frac{1}{\homega_2(\xi)+1}, \frac{\hlambda_2(\xi)}{\hlambda_2(\xi)+1}, \frac{\lambda_2(\xi)}{\lambda_2(\xi)+1} \Big).
    \end{align}
\end{proposition}

    Note that there also exists a parametric version of $\lambdaL(\xi)$ (see \cite[Section~3.2 and Proposition~3.6]{poels2019newExpo}), but we will not need it here.

    \medskip

    We follow \cite[§3]{Schmidt2013} and we define the \emph{combined graph} of a set of real valued functions defined on an interval $I$ to be the union of their graphs in $I\times\bR$. For a map $\bP:[c,+\infty)\rightarrow \bR^3$ and an interval $I\subset [c,+\infty)$, we also defined the \emph{combined graph of $\bP$ on $I$} to be the combined graph of its components $P_1,P_2,P_3$ restricted to $I$. In order to study the combined graph of the map $\bL_{\xi}$, it is useful to define the following functions.

\begin{definition}
    \label{Def L_y et L_z}
    For each point $\bx\in\bR^n\setminus\{0\}$ we denote by $\lambda_{\bx}(q)$ (resp. $\lambda_{\bx}^*(q)$) the smallest real number $\lambda>0$ such that $\bx\in\lambda\mathcal{C}_{\Xi}(e^q)$ (resp. $\bx\in\lambda\mathcal{C}_{\Xi}^*(e^q)$). Then, we set
    \[
        \CL_{\bx}(q) = \log(\lambda_{\bx}(q))\AND \CL_{\bx}^*(q) = \log(\lambda_{\bx}^*(q)).
    \]
    Roy calls the graph of $\CL_{\bx}$ (or of $\CL_{\bx}^*$) the \textsl{trajectory} of $\bx$.
\end{definition}

Locally, the combined graph of $\bL_{\xi}$ is included in the combined graph of a finite set $\mathrm{L}_{\bx}$, and for each $\bx\neq 0$ we have
\begin{align*}
    \CL_{\bx}(q) = \max\big(\log(\norm{\bx}), \log(|\bx\cdot\Xi|)+q\big) \AND \CL_{\bx}^*(q)
    = \max\big(\log(\norm{\bx\wedge\Xi}), \log(\norm{\bx})-q\big).
\end{align*}
Note that
\[
    \CL_1(q) = \min_{\bx\neq 0}\CL_{\bx}(q)\AND \CL_1^*(q) = \min_{\bx\neq 0}\CL_{\bx}^*(q).
\]

\begin{proposition}[Mahler]
    \label{Prop Mahler}
    For each $j=1,2,3$, we have $\pu_j = -\po_{4-j}$ and $\po_j=-\pu_{4-j}$. More precisely $\CL_{j}(q) = -\CL_{4-j}^*(q)+\GrO(1)$ for each $q>0$.
\end{proposition}

The functions $\CL_j$ have many rigid properties. For example they are continuous, piecewise linear with slopes $0$ and $1$, and by Minkowski's second Theorem, for any $q\geq 0$ we have
\begin{equation}
    \label{eq: Minkowski second thm}
    \CL_{1}(q)+\CL_2(q)+\CL_3(q) = q + \GrO(1).
\end{equation}
To describe precisely their behavior, several class of functions have been introduced, starting with the model of $(n,\gamma)$-systems of Schmidt and Summerer in \cite{Schmidt2013}. In \cite{poels2017exponents}, we use the simpler notion of $n$-system given by Roy in \cite{Roy_octobre}. The main result of \cite{Roy_juin} implies that $\bL_\xi$ can be approximated, up to an additive constant, by a $3$-system, and vice versa.

\subsection{Map of the successive minima} ~ \medskip
\label{subsection: construction 3-syst}

Let $\xi\in\Sturm(\us)$, where $\us$ is a sequence of positive integers. The goal of this section is to describe the map of successive minima $\bL_\xi$ and to determine its parametric exponents, see Proposition~\ref{Prop 3-system a o(q)} and Theorem~\ref{reciproque Thm exposants nb quasi sturmien} respectively. Our strategy is to construct a simpler and explicit function $\bP$ (very similar to that in \cite[Section~7.2]{poels2017exponents}) and show that $\bP(q)=\bL_\xi(q)+o(q)$, except in some controlled intervals which may be ignored for the computation of the parametric exponents, as was already the case in \cite{poels2017exponents}. Note that in \cite{poels2017exponents}, the sequence $\us$ is bounded and the parameter $\delta$ is $< \sigma/(1+\sigma)$. Here, $\us$ might be unbounded and we allow the case $\delta = \sigma/(1+\sigma)$, which brings some technical difficulties.

\medskip

Set $\psi=\psi_\us$ and $\sigma=\sigma(\us)$ and let $\delta\in[0,\sigma/(1+\sigma)]$, $(W_k)_{k\geq 0}$, $(Y_i)_{i\geq 0}$, $(Z_i)_{i\geq 0}$, and the sequences of primitive integer points $(\by_i)_{i\geq -1}$ and $(\bz_i)_{i\geq 0}$ in $\bZ^3$ be as in Theorem~\ref{reciproque Prop construction 3-systeme : construction bwu et byu}. Note that $W_k > 1$ for each $k\geq 0$ and that $(Y_i)_{i\geq 0}$ is increasing. The theory of continued fractions provides the useful formula (see \cite[Eq. (7.2)]{poels2017exponents})
\begin{equation}
    \label{eq: sigma via W_k}
    \sigma = \liminf_{k\rightarrow\infty} \frac{p_k}{p_{k+1}} = \liminf_{k\rightarrow\infty} \frac{\log W_k}{\log W_{k+1}}.
\end{equation}

\begin{definition}
Given $i=t_k+\ell$ with $k\geq 1$ and $0\leq \ell < s_{k+1}$, we denote by $\delta_i$ the maximum of the real numbers $\eta\leq \delta$ such that $Y_i^{1-\eta} \geq \max(Z_{t_{k+1}},Z_i)$. We set $q_{i} = (2-\delta_i)\log(Y_{i})$ and $c_i  = q_i + \log(W_k)$, as well as
\begin{equation*}
         \DD_{i}^* := \frac{Y_i^{\delta_i}}{Y_i} = \big(W_k^{\ell+1}W_{k-1}\big)^{-(1-\delta_i)} \AND
         \DD_{i} := \frac{Y_i^{\delta_i}}{Y_{i+1}} = \frac{\DD_{i}^*}{W_k}. 
\end{equation*}
The functions $ \hCL_i$ and $\hCL_i^*$ are defined for each $q \geq 0$ by
\begin{align*}
    \hCL_i(q) = \max\big(\log(Z_{i}),\log(\DD_{i})+q\big) \AND  \hCL_i^*(q)
    = \max\big(\log(\DD_{i}^*),\log(Y_{i})-q\big).
\end{align*}
\end{definition}

Since  $Y_i > \max(Z_{t_{k+1}},Z_i)$, we have $\delta_i\geq 0$. Note that $q_i=\log Y_{i}-\log\DD_{i}^* = \log Z_i-\log\DD_i$ is the point at which $\hCL_i$ and $\hCL_i^*$ change slope.

\begin{lemma}
    \label{lem: lem delta_i = delta+o(1)}
    The sequence $(\delta_i)_{i\geq 0}$ converges to $\delta$ and $q_i < c_i< q_{i+1}$ for each large enough $i$.
\end{lemma}

\begin{remark}
If $\delta < \sigma/(1+\sigma)$, then we have $\delta_i = \delta$ for each $i$ large enough by \cite[Proposition 7.17]{poels2017exponents}. However, this might not be true if $\delta = \sigma/(1+\sigma) > 0$.
\end{remark}

\begin{proof}
    Let $i=t_k+\ell$ with $k\geq 1$ and $0\leq \ell < s_{k+1}$. Recall that $Y_{i+1}=W_k Y_i$. If $\delta = 0$, then $\delta_i = 0$ for each $i\geq 0$ and we obtain $q_{i+1} = q_i+ 2\log W_k > c_i > q_i$. Suppose now that $\delta > 0$, and therefore $\sigma > 0$. Using the inequality $\log W_{k-1} \geq (\sigma-o(1)) \log W_k$ coming from \eqref{eq: sigma via W_k}, we find
    \begin{align*}
        \log Y_i - (1+\sigma)\log Z_{t_{k+1}} = (\ell+1) \log W_k + \log W_{k-1}-(1+\sigma)\log W_k  & \geq o(\log W_k).
    \end{align*}
    We deduce from the above and $1-\delta\geq 1/(1+\sigma)$ that $Y_i^{1-\delta+o(1)} \geq Z_{t_{k+1}}$. Similarly,
    \begin{align*}
        Y_i - (1+\sigma)\log Z_i &= (\ell+1)\log W_k+\log W_{k-1} - (1+\sigma)(\ell\log W_k+\log W_{k-1}) \\
        & \geq (s_{k+1}+1)\log W_k+\log W_{k-1} - (1+\sigma)(s_{k+1}\log W_k+\log W_{k-1})\\
        & = \log W_k - \sigma\log W_{k+1} \geq  o(\log W_{k+1}),
    \end{align*}
    from which we get $Y_i^{1-\delta} \geq Z_{i}W_{k+1}^{o(1)}$. The sequence $\us$ is bounded since $\sigma > 0$, so that $o(s_{k+1}) = o(1)$ and $W_{k+1}^{o(1)} = Y_i^{o(1)} = W_k^{o(1)}$. We thus have $Y_i^{1-\delta+o(1)} \geq \max(Z_{t_{k+1}},Z_i)$, hence $\delta_i=\delta+o(1)$. As a consequence if $k$ is large enough, then
    $q_{i+1} = q_i+ (2-\delta+o(1))\log W_k > c_i > q_i$.
\end{proof}

\begin{lemma}
    \label{lem: combined graph partiel}
    There exists an index $i_0\geq 0$ such that for each $i=t_k+\ell \geq i_0$ with $k\geq 1$ and $0\leq \ell < s_{k+1}$, the combined graph of $\hCL_{t_{k+1}}$, $-\hCL_{i}^*$, $\hCL_{i}$ on $[c_{i-1},c_{i}]$ is as on Figure~\ref{figure 3system_partiel1.png}. Furthermore
    \begin{align}
    \label{eq:lem: pour description combined graph P}
    \hCL_i^*(c_i) = \hCL_{i+1}^*(c_i) \AND \hCL_i(c_i) =
    \left\{\begin{array}{ll}
    \hCL_{i+1}(c_i) & \textrm{if $\ell < s_{k+1}-1$} \\
    \hCL_{t_{k+2}}(c_i) & \textrm{if $\ell = s_{k+1}-1$}.
    \end{array}\right.
\end{align}
\end{lemma}

\begin{figure}[H] 

    \tikzstyle{Pointilles}=[color=black!50, dashed] 
    \tikzstyle{Pointilles2}=[dotted] 
    \tikzstyle{ZoneI}=[color=black!30, fill=black!15]
    \tikzstyle{Interv}=[|-|,line width=1pt]

    \begin{tikzpicture}[scale=0.8]


        \fill[ZoneI] (1,3.5) -- (2.5,3.5) -- (1,2) -- cycle;
        \fill[ZoneI] (8,5) -- (9,6) -- (9,5) -- cycle;

        \draw (1,1)-- (4,1)-- (9,6); 
        \draw [Pointilles2] (1,1)-- ++(-0.3,0);
        \draw [Pointilles2] (18-9,6)-- ++ (0.3,0.3);
        \draw (10.6-9,0.6) node {$\hCL_{t_k}$};

        \draw (10-9,2)-- (13-9,5)-- (18-9,5); 
        \draw [Pointilles2] (10-9,2)-- ++(-0.3,-0.3);
        \draw [Pointilles2] (18-9,5)-- ++(0.3,0);
        \draw (10.8-9,2.3) node {$-\hCL_{t_k}^*$};

        \draw (1,3.5)-- (9,3.5); 
        \draw [Pointilles2] (1,3.5)-- ++(-0.3,0);
        \draw [Pointilles2] (9,3.5)-- ++ (0.3,0);
        \draw (10.3-9,3.9) node {$\hCL_{t_{k+1}}$};


        \draw [Pointilles] (1,3.5)-- (1,0) node [below, color=black] {$c_{t_k-1}$};
        \draw [Pointilles] (2.5,3.5)-- (2.5,0) node [below, color=black] {$a_{t_k}$};
        \draw [Pointilles] (4,5)-- (4,0) node [below, color=black] {$q_{t_k}$};
        \draw [Pointilles] (6.5,3.5)-- (6.5,0) node [below, color=black] {$d_{k}$};
        \draw [Pointilles] (8,5)-- (8,0) node [below, color=black] {$b_{t_k}$};
        \draw [Pointilles] (9,6)-- (9,0) node [below, color=black] {$c_{t_k}$};


        \draw[Interv] (2.5,0.2)-- (8,0.2) node [midway,above] {$I_{t_k}$};

        \draw[line width=1pt] (10.5,7)-- (10.5,-0.5);


        \fill[ZoneI] (12,2) -- (13.5,3.5) -- (12,3.5) -- cycle;
        \fill[ZoneI] (16.5,5) -- (18,5) -- (18,6.5) -- cycle;

        \draw (12,1)-- (18,1); 
        \draw [Pointilles2] (12,1)-- ++(-0.3,0);
        \draw [Pointilles2] (18,1)-- ++ (0.3,0);
        \draw (17.5,1.5) node {$\hCL_{t_{k+1}}$};

        \draw (12,3.5)-- (15,3.5)-- (18,6.5); 
        \draw [Pointilles2] (12,3.5)-- ++(-0.3,0);
        \draw [Pointilles2] (18,6.5)-- ++ (0.3,0.3);
        \draw (17.5,6.3) node {$\hCL_{i}$};

        \draw (12,2)-- (15,5)-- (18,5); 
        \draw [Pointilles2] (12,2)-- ++(-0.3,-0.3);
        \draw [Pointilles2] (18,5)-- ++ (0.3,0);
        \draw (17.5,4.5) node {$-\hCL_{i}^*$};

        \draw [Pointilles] (12,3.5)-- (12,0) node [below, color=black] {$c_{i-1}$};
        \draw [Pointilles] (13.5,3.5)-- (13.5,0) node [below, color=black] {$a_{i}$};
        \draw [Pointilles] (15,5)-- (15,0) node [below, color=black] {$q_{i}$};
        \draw [Pointilles] (16.5,5)-- (16.5,0) node [below, color=black] {$b_{i}$};
        \draw [Pointilles] (18,6.5)-- (18,0) node [below, color=black] {$c_{i}$};


        \draw[Interv] (13.5,0.2)-- (16.5,0.2) node [midway,above] {$I_{i}$};

    \end{tikzpicture}

    \caption{\label{figure 3system_partiel1.png}
    Combined graph of $\hCL_{t_{k+1}}$, $-\hCL_{i}^*$, $\hCL_{i}$ on $[c_{i-1},c_{i}]$ with $t_k< i < t_{k+1}$}
\end{figure}

\begin{proof}
    By Lemma~\ref{lem: lem delta_i = delta+o(1)} we can suppose $i$ large enough so that $c_{j-1} < q_j < c_{j}$ for each $j\geq i$. Recall that $q_i$ is the point at which $\hCL_i$ and $\hCL_i^*$ change slope and that $Y_{i+1}=W_kY_i$. The intersection point abscissa of $\hCL_i^*$ and $\hCL_{i+1}^*$ is $\log Y_{i+1} - \log \DD_i^* = c_i$. If $\ell < s_{k+1}-1$ (resp. $\ell = s_{k+1}-1$) then $Z_i < Z_{i+1} = Y_i$ (resp. $Z_i < Z_{t_{k+2}} = Y_i$) and the intersection point abscissa of $\hCL_i$ and $\hCL_{i+1}$ (resp. $\hCL_i$ and $\hCL_{t_{k+2}}$) is $\log Y_i - \log \DD_i = c_i$. Hence \eqref{eq:lem: pour description combined graph P}.

    \medskip

    We now prove the first part of the lemma. It suffices to compare $\hCL_{t_{k+1}}(q)$, $-\hCL_{i}^*(q)$ and $\hCL_{i}(q)$ at $q=c_{i-1}$, $q_i$ and $c_i$. By definition of $\delta_i$, we have
    \begin{align*}
        (1-\delta_i)\log Y_i = -\hCL_i^*(q_i) \geq \max(\hCL_i(q_i),\hCL_{t_{k+1}}(q_i)) =
        \left\{
        \begin{array}{ll}
            \hCL_i(q_i) = \log Z_i & \textrm{if $i > t_k$} \\
            \hCL_{t_{k+1}}(q_i) = \log W_k & \textrm{if $i = t_k$}.
        \end{array}
        \right.
    \end{align*}
    Since  $-\hCL_i^*$, $\hCL_{t_{k+1}}$ are constant on $[q_i,c_i]$, we deduce that $-\hCL_i^*(c_i) \geq \hCL_{t_{k+1}}(c_i)$. Moreover
    \begin{align*}
        \hCL_{i}(c_i) = \hCL_{i}(q_i)+c_i-q_i = \log Y_i \geq (1-\delta_i)\log Y_i = -\hCL_i^*(c_i) \geq  \hCL_{t_{k+1}}(c_i).
    \end{align*}
    If $i > t_k$, then by the above, we have $\hCL_{i-1}(c_{i-1}) \geq -\hCL_{i-1}^*(c_{i-1}) \geq \hCL_{t_{k+1}}(c_{i-1})$. Combined with \eqref{eq:lem: pour description combined graph P}, this is equivalent to
    \begin{align*}
        \hCL_{i}(c_{i-1}) \geq -\hCL_{i}^*(c_{i-1}) \geq \hCL_{t_{k+1}}(c_{i-1}).
    \end{align*}
    Similarly, for $i=t_{k-1}+s_k-1 = t_k-1$, Eq. \eqref{eq:lem: pour description combined graph P} together with $\hCL_{i}(c_{i}) \geq -\hCL_{i}^*(c_{i}) \geq \hCL_{t_{k}}(c_{i})$ yields
    \begin{align*}
        \hCL_{t_{k+1}}(c_{t_k-1}) \geq -\hCL_{t_k}^*(c_{t_k-1}) \geq \hCL_{t_{k}}(c_{t_k-1}).
    \end{align*}

\end{proof}

\begin{definition}
\label{reciproque Def bP}
Let $i_0\geq 0$ satisfying Lemma~\ref{lem: combined graph partiel}. We define the function $\bP=(\CP_1,\CP_2,\CP_3)$ on $[c_{i_0-1},+\infty)$ as follows. For integers $i,k\geq 0$ with $i \geq i_0$ and $t_{k}\leq i < t_{k+1}$, we set for each $q\in(c_{i-1},c_i]$
\begin{equation*}
    \bP(q) := \Phi\Big(\hCL_{t_{k+1}}(q),-\hCL_{i}^*(q),\hCL_{i}(q)\Big),
\end{equation*}
where $\Phi:\bR^3\rightarrow\bR^3$ is the function which lists the coordinates of a point in monotonically increasing order. For each $i \geq i_0$ we denote by $I_{i} = [a_{i},b_{i}]\ni q_{i}$ the subinterval of $[c_{i-1},c_{i}]$ on which $\CP_3 = -\hCL_{i}^*$, and we set $I_{i}'=[b_{i},a_{i+1}]\ni c_i$, see Figure~\ref{figure 3system_partiel1.png}.
\end{definition}

If $\delta_i=0$, then $I_{i}'=\{c_i\}$. By Lemma~\ref{lem: combined graph partiel} the function $\bP$ is continuous, $\CP_1\leq \CP_2\leq\CP_3$ and
\begin{equation}
\label{eq: Minkowski pour P}
    \CP_1(q)+\CP_2(q)+\CP_3(q) = q
\end{equation}
for each $q\geq c_{i_0}$. More generally, we can show that $\bP$ is a $3$-system on $[c_{i_0},+\infty)$ (as defined in \cite[Definition 7.9]{poels2017exponents}), whose combined graph is as that in Figure~\ref{figure 3system_partiel2.png}.

\begin{figure}[H] 
  \centering
  \tikzstyle{ZoneI}=[color=black!30, fill=black!20]
  \tikzstyle{DelimitBlock}=[color=black!50, dashed] 
  \tikzstyle{Absc}=[color=black!50, dashed] 
  \tikzstyle{Interv}=[|-|,line width=1pt]
  \tikzstyle{Interv2}=[line width=0.7pt]

  \tikzstyle{AbsPos}=[below]
  \tikzstyle{SegPos}=[above]

    \begin{tikzpicture}[scale=0.32]
        \clip(0,-2)  rectangle (50,20);

        \fill[ZoneI] (6,2.25) -- (6.25,2.25) -- (6.5,2.5) -- (6.25,2.5) -- cycle;
        \fill[ZoneI] (8.7,3.6) -- (9.1,3.6) -- (9.5,4) -- (9.1,4) -- cycle;
        \fill[ZoneI] (13.9,4.95) -- (14.45,4.95) -- (15,5.5) -- (14.45,5.5) -- cycle;
        \fill[ZoneI] (21.1,8.55) -- (22.05,8.55) -- (23,9.5) -- (22.05,9.5) -- cycle;
        \fill[ZoneI] (28.3,12.15) -- (29.65,12.15) -- (31,13.5) -- (29.65,13.5) -- cycle;
        \fill[ZoneI] (45,15.75) -- (46.75,15.75) -- (48.5,17.5) -- (46.75,17.5) -- cycle;

        \draw [] (4.6,1.5)-- (10.45,1.5)-- (12.95,4)-- (33.25,4)-- (42.75,13.5)-- (58.9,13.5)-- (62.9,17.5)-- (65,17.5);
        \draw [Absc] (10.45,1.5)-- (10.45,0) node {$\shortmid$} node [AbsPos, color=black] {$q_{t_k}$};
        \draw [Absc] (33.25,4)-- (33.25,0) node {$\shortmid$} node [AbsPos, color=black] {$q_{t_{k+1}}$};
        \draw [Absc] (58.9,13.5)-- (58.9,0) node {$\shortmid$} node [AbsPos, color=black] {$q_{t_{k+2}}$};

        \draw [] (4.6,1)-- (4.75,1)-- (6,2.25)-- (6.25,2.25)-- (6.5,2.5)-- (7.6,2.5)-- (8.7,3.6)-- (9.1,3.6)-- (9.5,4)-- (12.95,4)-- (13.9,4.95)-- (14.45,4.95)-- (15,5.5)-- (18.05,5.5)-- (21.1,8.55)-- (22.05,8.55)-- (23,9.5)-- (25.65,9.5)-- (28.3,12.15)-- (29.65,12.15)-- (31,13.5)-- (42.75,13.5)-- (45,15.75)-- (46.75,15.75)-- (48.5,17.5)-- (62.9,17.5)-- (65,19.6);


        \draw [Interv] (8.7,0.1)  -- (9.5,0.1);
        \draw [Interv] (13.9,0.1) --  (15,0.1) node [midway,SegPos] {$I'_{t_k}\,\,$};
        \draw [Interv] (21.1,0.1) -- (23,0.1) node [midway,SegPos] {$I'_{t_k+1}$};
        \draw [Interv] (28.3,0.1) -- (31,0.1) node [midway,SegPos] {$I'_{t_{k}+2}$};
        \draw [Interv] (45,0.1) -- (48.5,0.1) node [midway,SegPos] {$I'_{t_{k+1}}$};


        \draw [|-,Interv2] (6.5,0.1) -- (8.7,0.1)  ;
        \draw [Interv2] (9.5,0.1) -- (13.9,0.1)  node [midway,SegPos] {$I_{t_k}$};
        \draw [Interv2] (15,0.1) -- (21.1,0.1) node [midway,SegPos] {$I_{t_k+1}$};
        \draw [Interv2] (23,0.1) -- (28.3,0.1) node [midway,SegPos] {$I_{t_{k}+2}$};
        \draw [Interv2] (31,0.1) -- (45,0.1) node [midway,SegPos] {$I_{t_{k+1}}$};

        \draw [Absc] (14.45,4.95) -- (14.45,0) node {$\shortmid$} node [AbsPos, color=black] {$c_{t_k}\,\,$};
        \draw [Absc] (18.05,5.5)-- (18.05,0)  node {$\shortmid$} node [AbsPos, color=black] {$q_{t_k+1}$};
        \draw [Absc] (22.05,8.55)-- (22.05,0) node {$\shortmid$} node [AbsPos, color=black] {$c_{t_{k}+1}$};
        \draw [Absc] (25.65,9.5)-- (25.65,0) node {$\shortmid$}  node [AbsPos, color=black] {$q_{t_k+2}$};
        \draw [Absc] (29.65,12.15)-- (29.65,0) node {$\shortmid$} node [AbsPos, color=black] {$c_{t_{k}+2}$};
        \draw [Absc] (46.75,15.75)-- (46.75,0) node {$\shortmid$} node [AbsPos, color=black] {$c_{t_{k+1}}$};

        \draw [] (4.6,2.1)-- (4.75,2.25)-- (6,2.25)-- (6.25,2.5)-- (6.5,2.5)-- (7.6,3.6)-- (8.7,3.6)-- (9.1,4)-- (9.5,4)-- (10.45,4.95)-- (13.9,4.95)-- (14.45,5.5)-- (15,5.5)-- (18.05,8.55)-- (21.1,8.55)-- (22.05,9.5)-- (23,9.5)-- (25.65,12.15)-- (28.3,12.15)-- (29.65,13.5)-- (31,13.5)-- (33.25,15.75)-- (45,15.75)-- (46.75,17.5)-- (48.5,17.5)-- (58.9,27.9)-- (65,27.9);

    \end{tikzpicture}
    \caption{\label{figure 3system_partiel2.png}
    Combined graph of $\bP$ (with $s_{k+1}=3$)}
\end{figure}

Using the estimate $\delta_i=\delta + o(1)$ together with \eqref{eq:thm estimations points primitiv 1}, we get the following estimates
\begin{equation}
    \label{prop: trajectoire y_i, z_i et composantes de P}
    \norm{\byt_i} = Y_i^{1+o(1)}, \quad  \norm{\byt_{i}\wedge\Xi} = \DD_{i}^*Y_i^{o(1)},
    \quad  \norm{\bzt_{i}} = Z_i^{1+o(1)} \AND |\bzt_{i}\cdot\Xi| = \DD_{i}Y_i^{o(1)},
\end{equation}
as $i$ tends to infinity. They play a crucial role in the proof of our next result.

\begin{proposition}
\label{Prop 3-system a o(q)}
Let $\xi\in\Sturm(\us)$ and denote by $\bP$ the function associated to $\xi$ as in Definition~\ref{reciproque Def bP}. Set $I:= \bigcup_{i \geq i_0} I_i$ and $I':= \bigcup_{i \geq i_0}I'_{i}$. Then
\begin{enumerate}[label=\rm(\roman*)]
    \item \label{reciproque enum 2 Thm 3-systeme nombre quasi sturmien}
        As $q\in I$ tends to infinity, we have $\bL_{\xi}(q) = \bP(q) + o(q)$;
    \smallskip
    \item \label{reciproque enum 3 Thm 3-systeme nombre quasi sturmien}
        As $q\in I'$ tends to infinity, we have $\CL_1(q) = \CP_1(q) + o(q)$ and
        \begin{align}
        \label{eq prop:reciproque enum 3}
        \CP_2(q) \leq \CL_2(q)+o(q) \leq \frac{\CP_2(q)+\CP_3(q)}{2} \leq \CL_3(q)+o(q) \leq \CP_3(q).
        \end{align}
\end{enumerate}
In particular, if $\delta=0$, then $\bL_{\xi}(q) = \bP(q) + o(q)$ as $q$ tends to infinity.
\end{proposition}

Roughly speaking, the combined graph of $\CL_2$ and $\CL_3$ on $I_j'$ is included --within $o(c_{j})$ -- in the corresponding shaded area on Figure~\ref{figure 3system_partiel2.png}. Our strategy is very similar to that in \cite[proof of Proposition~7.20]{poels2017exponents}. Here, the situation is a bit more complicated because $\us$ can be unbounded and we deal with some $o(q)$ instead of $\GrO(1)$.

\begin{proof}
Since $Y_{i+1}^{o(1)} = Y_i^{o(1)}$ as $i$ tends to infinity, we have $o(q_{i+1}) = o(q_i)$ and thus $o(a_i) = o(b_i) = o(a_{i+1})$. Note that if $\delta=0$, then $I_i'=\{c_i\}$ and \ref{reciproque enum 2 Thm 3-systeme nombre quasi sturmien} implies \ref{reciproque enum 3 Thm 3-systeme nombre quasi sturmien} and the last part of the proposition.

\medskip

Suppose that \ref{reciproque enum 2 Thm 3-systeme nombre quasi sturmien} holds, and let us prove \ref{reciproque enum 3 Thm 3-systeme nombre quasi sturmien}. Recall that $\CL_1$, $\CL_2$, $\CL_3$ are (continuous) piecewise linear with slope $0$ or $1$. In particular, they  are monotonically increasing. Let $i > 0$ and $q\in I_i'=[b_i,a_{i+1}]$. We deduce from \ref{reciproque enum 2 Thm 3-systeme nombre quasi sturmien} the estimates
\begin{align*}
\CP_1(b_i)+o(b_i) = \CL_1(b_i) \leq \CL_1(q) \leq \CL_1(a_{i+1}) = \CP_1(a_{i+1}) + o(a_{i+1}).
\end{align*}
Since $\CP_1$ is constant on the interval $I_i'$ and $o(a_{i+1})=o(b_i)$, we obtain $\CL_1(q) = \CP_1(q) + o(q)$. The function $\CP_3$ is increasing with slope $1$ on $[b_i,c_i]$, and constant on $[c_i,a_{i+1}]$. Since $\CL_3$ has slope at most $1$ and satisfies $\CL_3(b_i) = \CP_3(b_i)+o(b_i)$ and $\CL_3(a_{i+1}) = \CP_3(a_{i+1})+o(a_{i+1})$ by \ref{reciproque enum 2 Thm 3-systeme nombre quasi sturmien}, we get $\CL_3(q)\leq \CP_3(q)+o(q)$. Similarly, we find $\CL_2(q)\geq \CP_2(q)+o(q)$. Finally, by \eqref{eq: Minkowski second thm}, the estimate $\CL_1(q)=\CP_1(q)+o(q)$ and \eqref{eq: Minkowski pour P}, we find
\begin{align*}
\CL_2(q)+\CL_3(q) = q-\CL_1(q)+\GrO(1) = \CP_2(q)+\CP_3(q) + o(q).
\end{align*}
Combining the above with $\CL_2(q)\leq \CL_3(q)$, we obtain the remaining inequalities of \eqref{eq prop:reciproque enum 3}.

\medskip

We now prove \ref{reciproque enum 2 Thm 3-systeme nombre quasi sturmien}, or equivalently, that for each $\ee > 0$ there exists $i_1$ such that
\begin{align}
\label{eq proof:prop pour 3-system:eq3}
    \norm{\bL_{\xi}(q) - \bP(q)}\leq  \ee q
\end{align}
for each $q\in I_i$ with $i\geq i_1$. Let $i > i_0$ and $k\ge 0$ such that $t_k\leq i < t_{k+1}$. Eq. \eqref{prop: trajectoire y_i, z_i et composantes de P} implies that
\begin{align}
    \label{eq proof:prop pour 3-system}
    \bP(q) = \Phi\big(\CL_{\bzt_{t_{k+1}}}(q),\CL_{\bzt_i}(q),-\CL_{\byt_i}^*(q)\big)+o(q) \qquad (q\in I_i),
\end{align}
where $\Phi:\bR^3\rightarrow\bR^3$ is the function which lists the coordinates of a point in monotonically increasing order. The points $\bzt_i$ and $\bzt_{t_{k+1}}$ are linearly independent, for $\bzt_{t_{k+1}}\wedge \bzt_i = \pm \byt_i$ by definition. This implies that $\CL_1 \leq \min\{\CL_{\bzt_i},\CL_{\bzt_{t_{k+1}}}\}$ and $\CL_2 \leq \max \{ \CL_{\bzt_i},\CL_{\bzt_{t_{k+1}}}\}$. Similarly, we have $\CL_1^* \leq \CL_{\byt_i}^*$. Combined with \eqref{eq proof:prop pour 3-system} and $\CL_3 = -\CL_1^* + \GrO(1)$, we obtain, as $q\in I_i$ tends to infinity,
\begin{align}
\label{eq proof:prop pour 3-system:eq2}
\CL_1(q) \leq \CP_1(q)+o(q), \quad \CL_2(q) \leq \CP_2(q)+o(q) \AND \CL_3(q) \geq \CP_3(q)+o(q).
\end{align}
Fix $\ee_1,\ee_2 > 0$ and choose $q\in I_i$.

\medskip

\noindent\textbf{First case.} Suppose that $\CL_3(q) \leq \CP_3(q)+\ee_1 q$. By \eqref{eq: Minkowski second thm} and \eqref{eq: Minkowski pour P}  we have
\begin{align*}
    \CL_1(q)+\CL_2(q) +\GrO(1) = q -\CL_3(q) \geq q - \CP_3(q) - \ee_1 q = \CP_1(q)+\CP_2(q) - \ee_1 q.
\end{align*}
Combined with \eqref{eq proof:prop pour 3-system:eq2}, it shows that \eqref{eq proof:prop pour 3-system:eq3} holds with $\ee=2\ee_1$ if $i$ is large enough.

\medskip

\noindent\textbf{Second case.} Suppose that $\CL_3(q) \geq \CP_3(q)+\ee_1 q$. We claim that if $q$ is large enough, then necessarily $i=t_k$ and $|q-d_k| < \ee_2 d_k$ (where $d_k$ is defined by $\hCL_{t_{k+1}}(d_k) = \hCL_{t_{k}}(d_k)$, see Figure~\ref{figure 3system_partiel1.png}). Yet the components of $\bL_\xi$ and $\bP$ are continuous with slope $0$ or $1$. So, by using the first case with $q=(1-\ee_2)d_k$ and by taking $\ee_2$ sufficiently small, it yields \eqref{eq proof:prop pour 3-system:eq3} with $\ee = 3\ee_1$. We now prove our claim.

\medskip

If $q$ is large enough, and since $\CL_1^*=-\CL_3+\GrO(1)$ and $-\CP_3(q) = \CL_{\byt_i}^*(q)+o(q)$, we deduce the existence of  a non-zero primitive point $\by\in\bZ^3$ such that $\CL_1^*(q) = \CL_{\by}^*(q) < \CL_{\byt_i}^*(q)$. It follows that $\by$ and $\byt_i$ are linearly independent (since they are both primitive), hence $\CL_2^*(q)\leq \CL_{\byt_i}^*(q)$. Combined with Mahler's duality and \eqref{eq proof:prop pour 3-system}, we deduce that $\CL_2(q) \geq \CP_3(q)+o(q)$. In view of \eqref{eq proof:prop pour 3-system:eq2}, we obtain $\CL_2(q)=\CP_2(q)+o(q)=\CP_3(q)+o(q)$, from which we infer
\begin{align*}
    \CL_1(q) = q-\CL_2(q)-\CL_3(q)+\GrO(1) \leq q - \CP_2(q) - \CP_3(q) - \ee_1 q + o(q) = \CP_1(q)- \ee_1 q + o(q).
\end{align*}
Consequently, if $q$ is large enough, then $\CL_1(q) < \min\{\CL_{\bzt_i}(q),\CL_{\bzt_{t_{k+1}}}(q)\} = \CP_1(q)+o(q)$. Since $\bzt_i$ and $\bzt_{t_{k+1}}$ are both primitive points, we obtain
\begin{align*}
    \CL_2(q) \leq \min\{\CL_{\bzt_i}(q),\CL_{\bzt_{t_{k+1}}}(q)\} = \CP_1(q)+o(q).
\end{align*}
Hence
$\CP_1(q)=\CP_2(q)+o(q) = \CP_3(q)+o(q)$. Fix $\ee_3 > 0$. Then, by the above and \eqref{eq: Minkowski pour P}, there exists $i_1$ such that if $i\geq i_1$, then
\begin{align*}
\label{eq proof:prop pour 3-system:eq4}
    \big| \CP_j(q) - \frac{q}{3}\big| \leq \ee_3 q \quad (j=1,2,3).
\end{align*}
By taking $\ee_3$ small enough, we deduce that $i=t_k$ and $q\in[(1-\ee_2)d_k,(1+\ee_2)d_k]$, for $\CP_1$ is increasing with slope $1$ on $[q_{t_k},d_k]$, and $\CP_3\geq (\CP_2+\CP_3)/2$, which is increasing with slope $1/2$ on $[d_k,q_{t_{k+1}}]$ (see Figures \ref{figure 3system_partiel1.png} and \ref{figure 3system_partiel2.png}). This ends the proof of our claim.
\end{proof}

As a consequence we get the following result.

\begin{theorem}
\label{reciproque Thm exposants nb quasi sturmien}
Let $\xi\in \Sturm(\us)$ and denote by $\sigma,\delta$ the parameters associated to $\xi$ as at the beginning of this section. For each $i=1,2,3$, we denote by $\pu_i,\po_i$ the parametric exponents associated to $\Xi=(1,\xi,\xi^2)$ as in Section~\ref{subsection geom param des nb}. Then, we have
\begin{align*}
&\pu_1 = \frac{\sigma}{(2-\delta)(1+\sigma)},\quad \po_1 = \frac{1}{(1-\delta)(1+\sigma)+2}, \quad \po_2 = \frac{1}{2+\sigma},\\
&\pu_3 = \frac{(1-\delta)(1+\sigma)}{1+2(1-\delta)(1+\sigma)},\quad
\frac{1-\delta}{2-\delta} \leq \po_3 \leq \max\Big(\frac{1}{2-\delta+\sigma},\frac{1-\delta}{2-\delta}\Big).
\end{align*}
If $\delta$ satisfies the stronger condition $\delta< h(\sigma)$ with
$h(\sigma) = \frac{\sigma}{2}+1-\sqrt{\big(\frac{\sigma}{2}\big)^2+1}$, then
\[
    \po_3 = \frac{1-\delta}{2-\delta}\AND \lambdaL(\xi) = \frac{(1-\delta)(1+\sigma)}{2+\sigma}.
\]
The left-hand side equality still holds if $\delta = h(\sigma)$.
\end{theorem}

\begin{remark}
    Theorem~\ref{reciproque Thm exposants nb quasi sturmien} shows that the parameters $\sigma$ and
    \begin{align}
        \label{eq def: parametre delta}
        \delta(\xi):=\delta
    \end{align}
    depend only on $\xi$.
\end{remark}

\begin{proof}[Proof of Theorem~\ref{reciproque Thm exposants nb quasi sturmien}]
Recall that $\bP=(\CP_1,\CP_2,\CP_3)$ is the function introduced in Definition~\ref{reciproque Def bP}. We define the parametric exponents associated to $\bP$ by
\begin{equation*}
    \thu_i :=  \liminf_{q\rightarrow\infty} \frac{\CP_i(q)}{q} \AND \tho_i:= \limsup_{q\rightarrow\infty} \frac{\CP_i(q)}{q} \quad (i=1,2,3).
\end{equation*}
They are computed in \cite{poels2017exponents} when $\us$ is bounded and $\delta < \sigma/(1+\sigma)$. The expressions for the exponents $\pu_i, \po_i$ are obtained by using Proposition~\ref{Prop 3-system a o(q)} and by arguing exactly as in the proof of \cite[Theorem~7.2]{poels2017exponents}, so we will skip most of the details. Note that if $\us$ is unbounded, then $\delta=0$ and $\bL_\xi(q)=\bP(q)+o(q)$. This yields $(\pu_i,\po_i)=(\thu_i,\tho_i)$ for $i=1,2,3$.

\medskip

In general, Proposition~\ref{Prop 3-system a o(q)} implies that $\CL_1(q)/q = \CP_1(q)/q+o(1)$, from which we deduce $(\pu_1,\po_1)=(\thu_1,\tho_1)$. As $i$ tends to infinity, we also have
\begin{align}
    \label{eq: inter 0 thm expo kappa}
     \sup_{q\in I_i'} \frac{\CP_3(q)}{q} = \frac{\CP_3(c_i)}{c_i} \leq \frac{1}{2-\delta+\sigma} + o(1)
     \AND \sup_{q\in I_i} \frac{\CP_3(q)}{q} = \frac{\CP_3(q_i)}{q_i}  = \frac{1-\delta}{2-\delta} + o(1),
\end{align}
and the upper bound for $\po_3$ follows easily, since $\CL_3(q)\leq \CP_3(q)+o(q)$. From now on we focus solely on the exponent $\lambdaL(\xi)$. Suppose that $\delta < h(\sigma)$, or equivalently that $1/(2-\delta+\sigma) < (1-\delta)/(2-\delta)$. Note that $\us$ is bounded, since otherwise $\delta = \sigma = 0$. Fix $\psi$ with
\begin{equation*}
    \frac{1}{2-\delta+\sigma} < \psi < \po_3 = \frac{1-\delta}{2-\delta}.
\end{equation*}
Let us prove that the exponent $\lambdaL(\xi)$ can be computed by using only the points $(\byt_i)_{i\geq 0}$. For each non-zero $\by\in\bZ^3$, we denote by $q(\by):=\log \norm{\by}-\log \norm{\by\wedge\Xi}$ the abscissa at which $\CL_{\by}^*$ changes slope. With this notation, we have $r_i:=q(\byt_i)=q_i+o(q_i)$. Eq.~\eqref{prop: trajectoire y_i, z_i et composantes de P} yields
\begin{equation*}
    \frac{\CL_{\byt_i}^*(r_i)}{r_i} = -\frac{1-\delta}{2-\delta} + o(1),
\end{equation*}
from which we deduce that $\CL_{\byt_i}^*(r_i)\leq -\psi r_i$ for each large enough $i$. Conversely, let $\by\in\bZ^3$ be a non-zero primitive point satisfying $\CL_{\by}^*(q)\leq -\psi q$, where $q := q(\by)$. Proposition~\ref{Prop 3-system a o(q)} gives $\CL_{\by}^*(q) \geq \CL_1^*(q) = -\CL_3(q)+\GrO(1) \geq -\CP_3(q)+o(q)$. Combined with the left-hand side of \eqref{eq: inter 0 thm expo kappa} we obtain $q\notin \bigcup_{i>0}I'_{i}$ if $\norm{\by}$ (and thus $q$) is large enough. Then, there is an index $i\geq 0$ such that $q\in I_i$. Suppose now that $\by$ and $\byt_i$ are linearly independent. Then, since $\CL_1^*(q)=\CL_{\byt_i}^*(q)+o(q)$, we get $\CL_2^*(q) \leq \CL_{\by}^*(q)+o(q)$, hence
\begin{align*}
    \CL_2(q) = -\CL_2^*(q)+\GrO(1) \geq -\CL_{\by}^*(q) + o(q) \geq \psi q + o(q).
\end{align*}
Yet, $\psi > 1/(2+\sigma) = \po_2$, so, if $q$ is large enough, $\by$ is proportional to $\byt_i$. Finally, note that $\CL_{\by}^*(q) \leq -\psi q$ is equivalent to $\norm{\by\wedge\Xi} \leq \norm{\by}^{-\mu}$, where $\mu = \psi/(1-\psi)$. By the above, the sequence of primitive points $\by\in\bZ^3$ such that $\norm{\by\wedge\Xi} \leq \norm{\by}^{-\mu} $ coincides, up to a finite numbers of terms, with the sequence $(\byt_i)_{i\geq 0}$. We deduce by a classical reasoning (see for example \cite[Section 2]{poels2019newExpo}) that
\[
    \hlambda_{\mu}(1,\xi,\xi^2) = \liminf_{i\rightarrow\infty} -\frac{\log \norm{\byt_i\wedge\Xi}}{\log \norm{\byt_{i+1}}} =
    \liminf_{i\rightarrow\infty} -\frac{\log \DD_i^*}{\log Y_{i+1}}.
\]
Let us write $i=t_k+\ell$, with $k\geq 0$ and $0\leq \ell < s_{k+1}$. Since the quotient
\begin{align*}
    -\frac{\log \DD_i^*}{\log Y_{i+1}} = \frac{(1-\delta)\big((\ell+1)\log W_k + \log W_{k-1}\big) }{(\ell+2)\log W_k + \log W_{k-1}}
\end{align*}
is increasing with $\ell$, it is minimum for $\ell = 0$, and we get by \eqref{eq: sigma via W_k}
\[
    \hlambda_{\mu}(1,\xi,\xi^2) = \liminf_{k\rightarrow\infty} \frac{(1-\delta)\big(\log W_k + \log W_{k-1}\big) }{2\log W_k + \log W_{k-1}}
    = \frac{(1-\delta)(1+\sigma)}{2+\sigma}.
\]
We deduce the value of $\lambdaL(\xi)$ from the above by noticing that $\mu = \psi/(1-\psi)$ tends to $\po_3/(1-\po_3) = \lambda_2(\xi)$ as $\psi$ tends to $\po_3$.
\end{proof}

\subsection{Approximation by algebraic numbers of degree at most $2$} ~ \medskip
\label{subsection: exposants omega^*}

We keep the notation of Section~\ref{subsection: construction 3-syst} for $\us$, $\psi=\psi_\us$, $\sigma:=\sigma(\us)$. In this section, we prove the following result.

\begin{proposition}
\label{prop: omega^* = omega}
Let $\xi\in\Sturm(\us)$. Then
\[
    \omega_2^*(\xi) = \omega_2(\xi) \AND \homega_2^*(\xi) = \homega_2(\xi).
\]
\end{proposition}

By Theorem~\ref{reciproque Thm exposants nb quasi sturmien} and \eqref{Eq dico exposants} we get the following explicit formulas
\begin{equation}
    \label{eq: omega(Xi) et homega(Xi)}
    \omega_2(\xi) = \frac{2-\delta}{\sigma}+1-\delta\AND \homega_2(\xi) = (1-\delta)(1+\sigma)+1.
\end{equation}

Before to prove Proposition~\ref{prop: omega^* = omega}, let us explain our strategy. First, by \eqref{eq: intro general pour exposants polynomials} we have the general estimates, valid for any real number $\xi$ which is neither rational nor quadratic
\begin{equation}
    \label{eq: inter 1 calcul omega^*}
    \omega_2^*(\xi) \leq \omega_2(\xi) \AND \homega_2^*(\xi) \leq \homega_2(\xi).
\end{equation}
By definition of $\omega_2^*(\xi)$ and $\omega_2^*(\xi)$, the reverse inequalities of \eqref{eq: inter 1 calcul omega^*} hold if the best solutions $P\in\bZ^3 \cong \bR[X]_{\leq 2}$ of the problems defining $\omega_2(\xi)$ and $\homega_2(\xi)$ (see Section~\ref{section: notation}), have two real roots $\alpha$, $\alpha'$, with $|\alpha'-\xi| \asymp 1$. Indeed, in that case we have $|\alpha - \xi| \asymp |P(\xi)|/H(\alpha)$. Here, Proposition~\ref{Prop 3-system a o(q)} indicates that the relevant polynomials to be considered correspond to the points $\bzt_{t_k}$ of Section~\ref{subsection: construction 3-syst}.

\begin{lemma}
    \label{lem: lemme racines}
Let $(\bw_k)_{k\geq 0}$ be an admissible $\psi$-Sturmian sequence with multiplicative growth. We denote by $(\bb_i)_{i\geq -1}$ and $(\ba_i)_{i\geq -1}$ the sequences of symmetric matrices associated by Definition~\ref{def: def alternative y_i et z_i}. We suppose that $(\ba_i)_{i\geq -1}$ converges projectively to a symmetric matrix $M_\xi$ identified with $(1,\xi,\xi^2)$, where $\xi$ neither rational nor quadratic. Then, there exist two distinct non-zero real numbers $\xi',\xi''\in\bR\setminus\{\pm\xi\}$ with the following properties. The point $\bb_i$ converges projectively to $P_0=(X-\xi)(X-\xi')$ (resp. $P_1=(X-\xi)(X-\xi'')$) as $i=t_k+\ell$ tends to infinity with $k\geq 0$ even (resp. odd) and $0\leq \ell < s_{k+1}$.
\end{lemma}

\begin{proof}
    We write $\Xi=(1,\xi,\xi^2)$ and denote by $N$ the matrix $(\Id-\bw_1^{-1}\bw_0^{-1}\bw_1\bw_0)J$. Proposition~\ref{Prop nouvelle expression y_i et z_i} gives, for any $k,\ell$ with $k\geq 1$ and $0\leq \ell < s_{k+1}$,
    \begin{equation*}
        \bb_{t_k+\ell}  =  U(\bw_k^\ell\bw_{k-1}) = U(\ba_{\psi(t_k+\ell)}N_k^{-1})
    \end{equation*}
    (see \eqref{Eq y_psi(k) = w_k-1} for the case $\ell=0$). We deduce that $\bb_{t_k+\ell}$ converges projectively to $Q_0:=U(M_\xi N^{-1})$ (resp. $Q_1:=U(M_\xi \transpose{N}^{-1})$) as $i=t_k+\ell$ tends to infinity with $k$ even (resp. $k$ odd) and $0\leq \ell < s_{k+1}$. Now, let us identify $Q_0,Q_1$ and $\bb_i$ with their corresponding polynomial of degree $\leq 2$ as in Section~\ref{section: notation}. Explicitly, we have the formulas
    \begin{align*}
        \left\{
            \begin{array}{ll}
                Q_0 &= -\xi(a+c\xi) + (a+(c-b)\xi-d\xi^2)X +(b+d\xi)X^2\\
                Q_1 &= -\xi(a+b\xi) + (a+(b-c)\xi-d\xi^2)X +(c+d\xi)X^2
            \end{array}
        \right.,
        \quad \textrm{where } N^{-1} = \left(\begin{array}{cc} a & b\\ c & d \end{array}\right).
    \end{align*}
    Recall that $N$ is inversible, neither symmetric nor antisymmetric by Lemma~\ref{lem: matrice N = J}, and $\xi$ is not the root of a polynomial in $\bZ[X]$ of degree $2$. This implies that none of the coefficients of $Q_0$ and $Q_1$ is zero. By the above, the discriminant $\Delta$ of $Q_0$ and $Q_1$ is equal to $\DD = (a+(b+c)\xi+d\xi^2)^2 > 0$. The two distincts roots of $Q_0$ (resp. $Q_1$) are $\xi$ and
    \[
        \xi':=-\frac{a+c\xi}{b+d\xi}\quad\Big(\textrm{resp. } \xi'':=-\frac{a+b\xi}{c+d\xi} \Big).
    \]
    and $\xi'\neq \xi''$ (since $N$ is neither symmetric nor antisymmetric) and $\xi',\xi''\notin \{0,-\xi\}$ since $Q_0$ and $Q_1$ have non-zero coefficients.
\end{proof}

\begin{proof}[Proof of Proposition~\ref{prop: omega^* = omega}]
We write $\Xi=(1,\xi,\xi^2)$ and we keep the notation of Section~\ref{subsection: construction 3-syst} for the $\psi$-Sturmian sequence $(\bw_k)_{k\geq 0}$, the parameter $\delta\in[0,\sigma/(1+\sigma)$ and the sequence of primitive integer points $(\bzt_i)_{i\geq -1}$ (also viewed as a sequence of polynomials) associated with $(\bb_i)_{i\geq -1}$ as in Theorem~\ref{reciproque Prop construction 3-systeme : construction bwu et byu}. Let $P_0,P_1$ and $\xi',\xi''$ be as in Lemma~\ref{lem: lemme racines}. Since $P_0$ and $P_1$ have non-zero coefficients and positive discriminant, each of the coordinates of $\bzt_i$ is $\asymp \norm{\bzt_i}$, and the discriminant of $\bzt_i$ is positive for large enough $i$. For those $i$, we denote by $r_i$ and $r_i'$ the two real roots of $\bzt_i$, where $r_i$ is the closest one to $\xi$. The other root $r_i'$ converges to either $\xi'$ (resp. $\xi''$) as $i=t_k+\ell$ tends to infinity with $k$ even (resp. odd) and $0\leq \ell < s_{k+1}$. Since $\xi',\xi''\neq \xi$, we have $|\xi-r_i'| \asymp 1$. Also note that the minimal polynomial of $r_i$ divide $\bzt_i$, and thus $H(r_i)\ll \norm{\bzt_i}$. This yields
\begin{equation}
    \label{eq proof: racine r_i}
    H(r_i) \ll \norm{\bzt_i} \AND |\xi-r_i| \asymp \frac{|\bzt_i\cdot \Xi|}{\norm{\bzt_i}} \ll \frac{|\bzt_i\cdot \Xi|}{H(r_i)}.
\end{equation}
Note that $H(r_i)$ tends to infinity since $r_i$ converges to $\xi$ which is neither rational nor quadratic. For each $k\geq 1$ let $\ee_k\geq 1$ be such that $W_k = W_{k-1}^{\ee_k}$, where $W_k$ is defined as in Section~\ref{subsection: construction 3-syst}, and consider $i:=t_k$. Then $H(r_{t_k})\leq W_{k-1}^{1+o(1)}$ and \eqref{prop: trajectoire y_i, z_i et composantes de P} leads us to
\begin{align*}
    |\bzt_{t_k}\cdot \Xi| = W_k^{-(2-\delta+o(1))}W_{k-1}^{-(1-\delta)} = W_{k-1}^{-((2-\delta+o(1))\ee_k+1-\delta)}
    & \leq H(r_{t_k})^{-((2-\delta+o(1))\ee_k+1-\delta)}.
\end{align*}
Since $\limsup_{k\rightarrow\infty}\ee_k = 1/\sigma$ by \eqref{eq: sigma via W_k}, the above combined with \eqref{eq proof: racine r_i} and \eqref{eq: omega(Xi) et homega(Xi)} yields
\[
    \omega_2^*(\xi) \geq \limsup_{k\rightarrow\infty}\big( (2-\delta+o(1))\ee_k+1-\delta \big)= \frac{2-\delta}{\sigma} +1-\delta = \omega_2(\xi).
\]
Similarly, let $X\geq 1$ be a large real number, and let $k$ be such that $H(r_{t_k}) \leq X < \max\{W_{k},H(r_{t_{k+1}})\} = W_k^{1+o(1)}$. Then, we find
\begin{align*}
    |\bzt_{t_k}\cdot \Xi| = W_k^{-(2-\delta+o(1))}W_{k-1}^{-(1-\delta)}  = W_{k}^{-(2-\delta+o(1)+\ee_k^{-1}(1-\delta))}
    \leq X^{-(2-\delta+o(1)+\ee_k^{-1}(1-\delta))}.
\end{align*}
Combining once again this result with \eqref{eq proof: racine r_i} and \eqref{eq: omega(Xi) et homega(Xi)}, we obtain
\[
    \homega_2^*(\xi) \geq \liminf_{k\rightarrow\infty} \big( 2-\delta+o(1)+\ee_k^{-1}(1-\delta)\big) = 2-\delta +(1-\delta)\sigma = \homega_2(\xi).
\]
\end{proof}

\subsection{Proofs}~\medskip
\label{subsection: proofs des thm}

\begin{proof}[Proof of Theorem~\ref{reciproque Thm 3-systeme nombre quasi sturmien intro}]
As seen in the introduction, $\Sturm(\us)$ is as most countable, and the density of $\Delta(\us)$ in $[0,\sigma/(1+\sigma)]$ when $\sigma > 0$ (or equivalently $\us$ bounded) comes from the construction of $\psi$-Sturmian numbers (see \cite[Section 9]{poels2017exponents}). The part concerning the exponents is a direct consequence theorem~\ref{reciproque Thm exposants nb quasi sturmien} combined with \eqref{Eq dico exposants} and Proposition~\ref{prop: omega^* = omega}. Since Bugeaud-Laurent continued fraction $\xi_\phi\in \Sturm(\us)$ with $\delta(\xi_\phi)=0$ we have $0\in\Delta_\us$.
\end{proof}

Corollary~\ref{cor: spec homega_2^* dense II} is a consequence of the following result. As defined in the introduction, $\usFibo=(s_k)_{k\geq 1}$ is the constant sequence $s_k=1$ for each $k\geq 1$.

\begin{lemma}
     There exists $\ee > 0$ with the following property. For each $\xi\in\bR$ which is neither rational nor quadratic, if $\homega_2(\xi)\geq \gamma^2-\ee$, then $\xi\in \Sturm(\usFibo)$.
\end{lemma}

\begin{proof}
The set of points $(1,\eta,\eta^2)$ with $\eta\in\bR$ corresponds to a quadratic hypersurface associated to the quadratic form $x_0x_2-x_1^2$. As a consequence of \cite{poelsroy2019} (see \cite[Theorem 7.3]{poelsroy2019}), for each $\eta$ with $0 < \eta < 1$, there is $\ee' > 0$ with the following property. Let $\xi\in\bR$ (which is neither rational nor quadratic) with $\hlambda_2(\xi)\geq 1/\gamma - \ee'$, and write $\Xi=(1,\xi,\xi^2)$. Then, there exists a sequence of primitive points $(\bv_i)_{i\geq 0}$ such that
\begin{enumerate}[label=\rm(\roman*)]
    \item The sequence $(\norm{\bv_i})_{i\geq 0}$ tends to infinity.
    \smallskip
    \item The matrix $\bv_{i+1}$ is  proportional to $\bv_i\Adj(\bv_{i-2})\bv_i$ for each large enough $i$.
    \smallskip
    \item \label{item: item proof cor 1} We have $\norm{\bv_i\wedge\Xi} \ll \norm{\bv_i}^{-1+\eta}$.
\end{enumerate}
The above phenomenon was first observed by Fischler in an unpublished work. Using the classical estimate $|\det(\bv_i)| \ll \norm{\bv_i}\norm{\bv_i\wedge\Xi}$, we deduce from \ref{item: item proof cor 1} that $\limsup_{i\rightarrow \infty} \log|\det(\bv_i)| / \log \norm{\bv_i}$ is at most $\eta$. So, if we choose $\eta \leq \sigma/(1+\sigma)$ (where $\sigma=\sigma(\usFibo)=1/\gamma$), we find $\xi\in\Sturm(\usFibo)$. By Jarn\'ik's identity \eqref{Eq Jarnik}, there exists $\ee > 0$ such that $\homega_2(\xi) \geq \gamma^2-\ee$ implies $\hlambda_2(\xi) \geq 1/\gamma -\ee'$.
\end{proof}

\begin{proof}[Proof of Theorem~\ref{Thm beta_0 < sqrt}]
Let $\xi\in\bR$ which is neither rational nor quadratic with $\beta_0(\xi) < \sqrt 3$. Let $(\bv_i)_{i\geq 0}$ and $\us$ be the sequences given by Theorem~\ref{Thm beta_0 < sqrt}. Since $\us$ is bounded, we have $\sigma=\sigma(\us) > 0$. The first two conditions of Definition~\ref{def: Sturm(psi)} are satisfied. The last one comes from the estimates $|\det(\by_i)| \ll \norm{\bv_i}\norm{\bv_i\wedge \Xi} =  \norm{\bv_i}^{o(1)}$ as $i$ tends to infinity, where $\Xi=(1,\xi,\xi^2)$. We therefore have $\xi\in\Sturm(\us)$ and by the above, the parameter $\delta(\xi)$ (see \eqref{eq def: parametre delta}) is equal to $0$, and thus $ < h(\sigma)$. Theorem~\ref{reciproque Thm 3-systeme nombre quasi sturmien intro} yields $\hlambda_2(\xi)=\lambdaL(\xi)$. We conclude by recalling that $\beta_0(\xi) < 2$ implies that $\lambdaL(\xi)=1/\beta_0(\xi)$ (see Section~\ref{section: notation}).
\end{proof}

\begin{proof}[Proof of Definition~\ref{Def fonction sturmienne, version polynomiale} $\Leftrightarrow$ Definition~\ref{def: Sturm(psi)}] The implication $\Leftarrow$ follows from Theorem~\ref{reciproque Prop construction 3-systeme : construction bwu et byu} and the estimates \eqref{prop: trajectoire y_i, z_i et composantes de P}.\\
$\Rightarrow$ Let $\xi$ and $(\bw_k)_{k\geq 0}$ be as in Definition~\ref{Def fonction sturmienne, version polynomiale}. The first two conditions ensure that $(\bw_k)_{k\geq 0}$ is admissible and has multiplicative growth. Note that $(\bwt_k)_{k\geq 0}$ (and thus $(\bw_k)_{k\geq 0}$) is unbounded since $P_k(\xi)$ tends to $0$ and $\xi$ is neither rational nor quadratic. By Proposition~\ref{Annexe Prop existence delta général}, there are $\alpha,\beta,\rho\geq 0$, with $\beta > 0$ and $\alpha\leq 2\beta$, such that $\norm{\bw_k} \asymp e^{\beta p_k}$, $|\det(\bw_k)|\asymp e^{\alpha p_k}$ and $c_k = e^{ p_k(\rho+o(1))}$. The condition \ref{item: Def sturm poly condition 4} can be rewritten as $(\alpha-2\rho)/(\beta-\rho) \leq \sigma/(1+\sigma)$. In particular, we must have $\alpha/\beta < 2$, and therefore Proposition~\ref{prop: estimations pour suites y_i et z_i} applies. We obtain that the sequence of symmetric matrices $(\by_i)_{i\geq -2}$ associated to $(\bwt_k)_{k\geq 0}$ converges projectively to a point $(1,\eta,\eta^2)$, which, by Lemma~\ref{lem: lemme racines} combined with condition \ref{item: Def sturm poly condition 3} of Definition~\ref{Def fonction sturmienne, version polynomiale}, is equal to $(1,\xi,\xi^2)$. Then, the sequence of primitive integer points $(\cont(\by_i)^{-1}\by_i)_{i\geq 0}$ is as in Definition~\ref{def: Sturm(psi)}.
\end{proof}

\Ack

\bibliographystyle{abbrv}

\end{document}